\documentclass[a4paper,12pt]{amsart}
\usepackage[T1]{fontenc}
\usepackage[utf8]{inputenc}
\usepackage[right=1.5cm,left=1.5cm,height=20cm]{geometry}
\usepackage{amsmath,amssymb,mathabx,verbatim}
\usepackage{tikz-cd}
\usepackage{caption}
\usepackage{array}
\usepackage{float}
\usepackage{lscape}
\usepackage{tikz}
\usepackage{amsthm}
\usepackage{cancel}
\usepackage{appendix}
\usepackage[
backend=biber,
style=alphabetic,
maxalphanames=4,
maxbibnames=4,
isbn=false,
url=false,
giveninits=true,
doi=false,
]{biblatex}
\usepackage{csquotes}
\AtBeginBibliography{\footnotesize}
\usepackage[bookmarks=true]{hyperref}
\usepackage{booktabs}
\usepackage[capitalise]{cleveref}
\usepackage{graphicx} 
\usepackage{xcolor}

\makeatletter
\@namedef{subjclassname@2020}{\textup{2020} Mathematics Subject Classification}
\makeatother

\DeclareUnicodeCharacter{0301}{\'{e}}

\newcommand{\QQ}{\mathbb Q}
\newcommand{\ZZ}{\mathbb{Z}}

\newcommand{\RR}{\mathbb R}
\newcommand{\FF}{\mathbb F}

\newcommand{\PP}{\mathbb P}

\newcommand{\End}{\mathrm{End}}

\newcommand{\Aut}{\mathrm{Aut}}
\newcommand{\GL}{\mathrm{GL}}
\newcommand{\Frob}{\mathrm{Frob}}
\newcommand{\GLn}{\mathrm{GL}_2(\mathbb Z/n\mathbb Z)}

\newcommand{\Id}{\mathop{\mathrm{Id}}}

\newcommand{\ns}{\textnormal{ns}}

\renewcommand{\theequation}{\arabic{section}.\arabic{equation}}

\newtheorem*{thm*}{Theorem} 
\newtheorem{thm}[equation]{Theorem}
\newtheorem*{prop*}{Proposition} 
\newtheorem*{lem*}{Lemma} 
\newtheorem{lem}[equation]{Lemma} 
\newtheorem*{cor*}{Corollary} 
\newtheorem{cor}[equation]{Corollary} 

\newtheoremstyle{named}{}{}{\itshape}{}{\bfseries}{.}{.5em}{#3}
\theoremstyle{named} 
\theoremstyle{remark}
\newtheorem*{rem*}{Remark} 
\newtheorem{ese}[equation]{Example} 
\theoremstyle{definition}
\newtheorem*{defi*}{Definition} 

\title{Maximal curves over finite fields and a modular isogeny}

\author{Valerio Dose}
\address{Valerio Dose, Department of Computer, Control and Management Engineering, Sapienza University of Rome, Roma, Italy}
\email{valerio.dose@uniroma1.it}
\author{Guido Lido}
\address{Guido Lido, Department of Mathematics, University of Rome ``Tor Vergata'', Roma, Italy}
\email{guidomaria.lido@gmail.com}
\author{Pietro Mercuri}
\address{Pietro Mercuri, Department of Mathematics, University of Trento, Trento, Italy}
\email{mercuri.ptr@gmail.com}
\author{Claudio Stirpe}
\address{Claudio Stirpe, Convitto Nazionale ``R. Margherita'', Anagni, Italy}
\email{clast@inwind.it}


\begin{document}
\begin{abstract}
We prove the existence of curves of genus $7$ and $12$ over the field with $11^5$ elements, reaching the Hasse-Weil-Serre upper bound.  
These curves are quotients of modular curves and we give explicit equations. 
We compute the number of points of many quotient modular curves in the same family without providing equations. For various pairs (genus, finite field) we find new records for the largest known number of points. In other instances we find quotient modular curves that are maximal, matching already known results.
To perform these computations, we provide a generalization of Chen's isogeny result.
\end{abstract}

\maketitle

\section{Introduction}
The Hasse-Weil-Serre bound prescribes that
\begin{equation}\label{eq:Hasse-Weil_bound}
|\# X(\mathbb F_{q})-q-1|\le g\lfloor 2\sqrt{q}\rfloor
\end{equation}
for every finite field $\mathbb F_q$ and every smooth, projective, absolutely irreducible algebraic curve $X$ over $\FF_q$ of genus $g$. 
We provide explicit equations for two smooth projective curves over $\QQ$, of genus $7$ and $12$ respectively, which have good reduction at $11$ and the largest possible number of points over the finite field with $11^5$ elements: they have  exactly $166666$ and $170676$ points $\FF_{11^5}$-points, respectively, achieving the Hasse-Weil-Serre bound.

The interest in explicit examples of curves with a number of points close to the Hasse-Weil-Serre bound is also motivated by their application to coding theory (see for example \cite[Section 8.4]{Sti09}).
The website \cite{manypoints} is devoted to collecting, for every choice of the genus $g\le 50$ and $q=p^k$ with $p<20$ a prime and $k\le 5$, the current bounds about how large $\# X(\mathbb F_q)$ can be.
Among the most recent efforts to populate this database, we mention \cite{Howe2017}, \cite{BGK20}, \cite{GMQ23}.

Certain families of modular curves have already been used in \cite{DLMS23} to produce many examples that improve the previously known lower bounds. The maximal curve of genus $12$ over $\mathbb F_{11^5}$ described in \Cref{sc:computing-equation} was found among those examples, even though explicit equations were not given.

In this work we extend the reach of the methods in \cite{DLMS23} to a larger family of modular curves, finding new records and several examples of maximal curves, among which the new maximal curve of genus $7$ over $\mathbb F_{11^5}$.

This instance and the one with genus $12$ are particularly surprising, because they arise over an odd extension of a field with a prime number of elements. Indeed, modular curves are generally known to have many points over even extensions because of the presence of rational points associated to supersingular elliptic curves, which do not all show over an odd extension. 
We compute explicit equations for the two maximal curves by studying the behavior of the automorphisms on the modular forms involved. 

In our study, we consider modular curves of Borel-Cartan type and their quotients by all subgroups made of Atkin-Lehner involutions or analogous involutions in the Cartan case. We refer to all these automorphisms as \emph{Atkin-Lehner type involutions}. 
Curves in this family have already been the subject of a lot of recent studies because of their relation to open questions in the arithmetic of elliptic curves, see for example \cite{Xnsp27}, \cite{FL23}, \cite{frengley2023congruences}, \cite{DLM22}, \cite{M-R22}, \cite{EP21}, \cite{BDMTV19}, \cite{Dos16}, \cite{KohenPacetti}, \cite{DFGS14}.
Quotients, in particular, can result in better instances from the point of view of the number of points with respect to the genus, as observed for example in \cite{CO25}.
Concerning the explicit computation of modular forms and explicit equations for modular curves, we mention the recent works \cite{AAL23}, \cite{LMFDBpaper}, \cite{Zyw21}, \cite{MS20},\cite{DMS19}, \cite{Mer18}, \cite{Bar13} as well as the more classical \cite{Shi95} and \cite{Gal99}. 

The paper is organized as follows. In \cref{sc:BC-curves} we recall the definition of modular curves of Borel-Cartan type, and we give an isogeny result of Chen type for the full family of those curves and their quotients by automorphisms of Atkin-Lehner type, generalizing \cite{Che98}, \cite{dSE00}, \cite{Che04}, \cite{DLM22} and \cite{DLMS23}. 
In \cref{sc:computing-points} we recall how to compute the number of points over finite fields, using the new isogeny result, and we list all the improved lower bounds. In \cref{sc:computing-equation} we 
compute equations for 
the two new maximal curve of genus 7 and 12. 
In \cref{sc:other}, we comment on all the maximal curves in the family in the range of our search. 
The results of our computations are summarized in the appendices: in \cref{app:bestresults} we report all the best results we found for each pair (genus, finite field),
while in \cref{app:maxcurves} we list all the maximal curves of genus larger than $1$, in the range of our search.

As an extra, for $X$ a curve with a subgroup $G \cong (\mathbb Z/2\mathbb Z)^r$ of the automorphism group, we give in \cref{sc:appendix_genus} some general results on the Jacobians of the quotient curves $X/W$, with $W$ a subgroup of $G$. From this, we deduce formulas that relate the genera and the number of points over finite fields of quotients of $X$ by different subgroups of $G$. These results can be relevant in the context of modular curves and their quotients, which in various cases have automorphism group isomorphic to $(\mathbb Z/2\mathbb Z)^r$ (see \cite{AS90}, \cite{Bar08},\cite{BH03} and \cite{DLM22}).

\subsection*{Acknowledgments}
We thank Jeremy Rouse for the helpful discussion and the whole LMFDB team for their remarkable work, which we also use here.

The second and the third authors are supported by the ``National Group for Algebraic and Geometric Structures, and their Applications" (GNSAGA - INdAM).
The second author is also supported the MIUR Excellence Department Project MatMod@TOV awarded to the Department of Mathematics, University of Rome Tor Vergata,  by the ``Programma Operativo (PON) “Ricerca e Innovazione” 2014-2020'' and by Prin PNRR 2022 Mathematical Primitives for Post Quantum Digital Signatures.

\section{Modular Curves of Borel-Cartan type and their quotients}\label{sc:BC-curves}

The curves we study are quotients of modular curves. Modular curves are smooth projective (algebraic) curves over $\QQ$ that parametrize elliptic curves with a torsion structure and whose complex points are somewhat easy to describe by quotienting the complex upper half-plane. We refer to \cite{DS05}, \cite{DR73} or to the small summary in the first section of \cite{DLM22}, for more details.

Given $n_0,n_\ns \in \ZZ_{>0}$ coprime and $n= n_0n_\ns$, let 
\[
X(n_0,n_\ns) = X_H,
\]
be the modular curve associated to 
\[
H= B_0(n_0) \times C_{\ns}(n_{\ns}) \quad \text{ inside } \quad \GLn = \GL_2(\ZZ/n_0\ZZ)\times \GL_2(\ZZ/n_\ns\ZZ),
\]
where $B_0$ is a choice of a Borel subgroup and $C_{\ns}$ is a choice of a non-split Cartan subgroup. By the general theory of modular curves, $X_H$ has good reduction over all primes $p\nmid n$ (see \cite{DR73} for example).
We say that these modular curves are of \textit{Borel-Cartan (mixed) type}, see \cite{DLMS23} for more details.  
The curves $X(n_0,n_\ns)$ also appeared in \cite{KohenPacetti} with the notation $X_{\ns}^{\Vec\varepsilon}(n_\ns,n_0)$, where $\Vec\varepsilon$ is a vector of non-square elements of $\ZZ/p\ZZ$ for each prime $p$ dividing $n_\ns$.

For every exact divisor $k$ of $n=n_0n_{\ns}$ (i.e., a divisor $k$ such that $k$ and $\frac nk$ are coprime), the curve $X(n_0,n_\ns)$ has an involution $W_k$  which we call \emph{Atkin-Lehner type} involution, defined as follows:
\begin{itemize}
\item For $k$ dividing $n_0$, we define $W_k$ as a lift of the classical Atkin-Lehner involution on $X_0$ as follows: we represent a non-cuspidal point on $X(n_0, n_\ns)$ as an elliptic curve $E$ together with cyclic subgroups $C_k$ of $E[k]$ and $C_{n_0/k}$ of $E[n_0/k]$ (they form the cyclic subgroup $C = C_k \oplus C_{n_0/k}$ of $E[n_0] = E[k] \oplus E[n_0/k]$) and together with an isomorphism $\phi \colon (\ZZ/n_\ns\ZZ)^2 \cong E[n_\ns]$ (up $C_\ns(n_\ns)$). Then we define  
\[
W_k(E, C_k, C_{n_0/k}, \phi) = (E/C_k, E[k]/C_k, \pi(C_{n_0/k}),  \pi\circ \phi )
\]
where $\pi$ is the projection $E \to E/C_k$.
\item For $k$ a prime power dividing $n_\ns$, we recall that under Chinese Remainder Theorem we have $C_{\ns}(n_\ns)  = C_{\ns}(k) \times C_{\ns}(n_\ns/k)$, and analogously for the normalizer $C_\ns^+(n_\ns)$. Moreover, since $k$ is a prime power, $C_\ns^+(k)/C_\ns(k)$ has two elements. Then we take the non-trivial element $w_k\in C_\ns^+(k)/C_\ns(k) < C_\ns^+(n_\ns)/C_\ns(n_\ns)$ and we define 
\[
W_k(E, C_{n_0}, \phi) = (E, C_{n_0}, \phi\circ w_k), 
\]
for all non-cuspidal points $(E, C_{n_0}, \phi)$ on $X(n_0, n_\ns)$. Moreover, in this case the quotient curve $X_H/\langle W_k\rangle$ is still of the form $X_{H'}$ for a suitable subgroup $H'$ of $\GLn$ containing $H$. 
\item For a general $k$ we extend the definition multiplicatively: each exact divisor $k$ is uniquely the product of an exact divisor $k_0\mid n_0$ and of pairwise coprime prime powers $q_i\mid n_\ns$ and we define 
$W_k = W_{k_0}\circ W_{q_1} \circ\cdots \circ W_{q_r}$ which does not depend on the order.
\end{itemize}
The above Atkin-Lehner type involutions have, indeed, order~$2$. They are defined over $\QQ$ (this follows from the definition of $X_H$ as coarse moduli space) and generate a subgroup of $\Aut(X(n_0,n_\ns))$ isomorphic to $(\ZZ/2\ZZ)^{r}$ where $r$ is the number of primes dividing $n_0n_\ns$. One can consider the quotients
\[
X(n_0,n_\ns)/\langle W_{k_1},\dots,W_{k_t}\rangle
\]
with $k_1,\dots,k_t$ dividing exactly $n$. In the following theorem we compute an isogeny decomposition of the Jacobian of these curves, generalizing \cite[Theorem 3.8]{DLMS23}.
\begin{thm}\label{pr:chen_quotients_all}
Given a curve $Y=X(n_0,n_\ns)$, let $W$ be the group of its Atkin-Lehner type involutions. For each subgroup $K<W$ the Jacobian of the corresponding quotient satisfies the following isogeny relation over $\QQ$
\[
\mathrm{Jac}(Y/K)\sim \textstyle\prod_f A_f^{m_f},
\]
where $f$ ranges among the newforms of weight $2$ and level $d_0d_{\ns}^2$ for divisors $d_0{\mid} n_0$ and $d_{\ns}{\mid} n_{\ns}$ and the multiplicities are as follows: let $K^\perp$ be the set of characters of $W$ that are trivial on $K$ and, for each $f$, denote $v:=\mathrm{val}_{p}(n_0/d_0)$, then
\[
m_f:=\sum_{\chi\in K^\perp} \left(  \prod_{p^e {\mid}{\mid} n_\ns}  \left(\tfrac12+\tfrac12\varepsilon_{f,p^e}\chi(W_{p^e}) \right) \cdot \prod_{p^e {\mid}{\mid} n_0} \left(\tfrac{v+1}{2}+\varepsilon_{f,p^e}\chi(W_{p^e})\tfrac{1+(-1)^{v}}{4} \right)  \right),
\]
with the indexes $p^e$ varying among all prime powers dividing exactly $n_0$ or $n_\ns$ and $\varepsilon_{f,p^e}$ being the eigenvalue of $W_{p^e}$ applied to $f$.
\end{thm}
\begin{proof}
    Let $A$ be the Jacobian of $Y$. Then $W$ acts on $A$ and the Jacobian of $Y/K$ is the connected component of $0$ in the fixed locus of $K$ in $A$, which we denote $A^K$.    
    By \cite[Theorem 3.8]{DLMS23} $A$ is isogenous to a product $\prod_f{A_f^{n_f}}$ with the same newforms $f$ appearing in the statement and $n_f$ suitably defined. Since the $A_f$'s are all simple and non-isogenous over $\QQ$, the action of $W$ on $A$ acts separately on each $A_f^{n_f}$. In particular the Jacobian of $Y/K$ is isogenous over $\QQ$ to the product of the fixed loci $(A_f^{n_f})^K$. Since the $A_f$'s are simple, each $(A_f^{n_f})^K$ is isogenous to a power $A_f^{m_f}$. We are left to compute the multiplicities $m_f$.

    Each subvariety $A_f^{n_f}$ is a module for the anemic Hecke algebra $\mathbb T$ (i.e., the algebra of Hecke operators whose indices are coprime with the level) of $A_f$ and the action of $\mathbb T$ commutes with $W$. In particular the space of $1$-forms on 
    $$V := H^0(A_f, \Omega^1)^{n_f}$$ 
    is a module both for $\mathbb T$ and for $W$, and the multiplicity $m_f$ is simply the rank, as $\mathbb T$-module, of the fixed locus $V^K$. 
    Since the irreducible representations of $W$ are characters, then $V$ has a $\mathbb T$-basis $f_1, \ldots, f_{n_f}$ with respect to which $W$ acts diagonally (such $f_i$ can be found explicitly extending \cite[Equations (5.1) and (5.2)]{AL70}):
    for each $i$, the group $W$ acts on $\mathbb T \cdot f_i$ though a certain character $\chi\colon W \to \{\pm1\}$. For each character $\chi$ we denote $d_\chi$ the number of $f_i$'s such that $W$ acts on $\mathbb T \cdot f_i$ with the character $\chi$. Then we have
    \begin{equation}\label{mfd_chi}
    m_f = \mathrm{rank}_{\mathbb T} V^K = \sum_{\chi\in K^\perp} d_\chi.
    \end{equation}
    Therefore, to prove our theorem it is enough to prove
    \begin{equation}\label{d_chi}
    d_\chi = d'_\chi, \quad \text{with }\quad d_\chi':= \prod_{p^e {\mid}{\mid} n_\ns}  \left(\tfrac12+\tfrac12\varepsilon_{f,p^e}\chi(W_{p^e}) \right) \cdot \prod_{p^e {\mid}{\mid} n_0} \left(\tfrac{v+1}{2}+\varepsilon_{f,p^e}\chi(W_{p^e})\tfrac{1+(-1)^{v}}{4}  \right).
    \end{equation}
    We prove it with \cite{DLMS23} and some linear algebra.
    
    Let $Q $ be the set of Atkin-Lehner type involutions $W_{p^e}$ for $p^e{\mid}{\mid} n_0 n_\ns$: then $Q$ generates $W$ and
    \[
    W \cong (\ZZ/2\ZZ)^Q \cong (\ZZ/2\ZZ)^{\#Q}.
    \]
Hence the characters of $W$ are in bijection with $S := \{\pm 1 \}^{\#Q}$ since each character of $W$ is determined by its value on each involution $W_{p^e}$. In particular the vectors $d := (d_\chi)_\chi$ and $d' := (d'_\chi)_\chi$ live in the vector space $\RR^S \cong \RR^{2^{\#Q}}$.
Now for each subset $Z \subset Q$ consider $\tilde Z \in \RR^S$ the vector only containing $0$'s or $1$'s as follows: for $s\in S$ we have $\tilde Z_s = 1$ if and only if $s$ satisfies $s(q) =1$  for each $q$ in $Z$. Then, by induction on $\#Q$, the set $\{\tilde Z: Z \subset Q\}$ is a basis for $\RR^S$. Hence, denoting by $\langle\cdot,\cdot \rangle$ the standard scalar product on $\RR^S$, determining the vector $d$ is equivalent to determining $\langle d, \tilde Z \rangle$ for each subset $Z\subset Q$. In particular to prove Equation \eqref{d_chi} it is enough to prove $\langle d, \tilde Z \rangle = \langle d', \tilde Z \rangle$ for each $Z$. By definition of $\tilde Z$ we have
\[
\langle d, \tilde Z \rangle = \sum_{\substack{\chi(W_q)=1 \\ \forall q\in Z}} d_\chi  =  
\sum_{\langle W_q:q\in Z\rangle^\perp} d_\chi
\]
which, by Equation \eqref{mfd_chi}, is the multiplicity of $A_f$ in $Y/\langle W_q:q\in Z\rangle$. This multiplicity is computed in \cite[Theorem 3.8]{DLMS23} and it is exactly equal to $\langle d', \tilde Z \rangle$.
\end{proof}

\section{Computing the number of points}\label{sc:computing-points}

The curve $X = X(n_0,n_\ns)/\langle W_{k_1},\dots,W_{k_t}\rangle$ is defined over $\QQ$ and has smooth reduction for all primes $p\nmid n_0n_{\ns}$, in fact the full level curve $X(n)$ has good reduction by \cite[Theorem IV.6.7]{DR73}, while \cite[Corollary 4.10]{liulorenzini} or \cite[Th\'{e}or\`{e}me 2.2]{You93} or \cite[Lemma 5.1]{NS96} take care of the quotient by $H$ and by $\langle W_{k_1},\dots,W_{k_t}\rangle$.

\Cref{pr:chen_quotients_all} allows to compute its number of points over a finite extension of $\FF_p$, starting from the knowledge of the newforms appearing in the decomposition of \Cref{pr:chen_quotients_all}.
In more detail, for $q$ a power of $p$, one can compute $\# X(\FF_{q})$  applying \cite[Algorithm 4.5]{DLMS23}, after substituting $m_f$ with the multiplicities in \Cref{pr:chen_quotients_all}.

Although we refer to \cite{DLMS23} for the algorithm, we show the procedure in a couple of examples.
\begin{ese}\label{ex_max}
Let's consider the curve 
$X = X(6,7)/\langle W_6,W_7\rangle$. We follow the notation of \Cref{pr:chen_quotients_all}. Listing the possible positive divisors $d_0$ of $6$ and $d_{\ns}$ of $7$ we have $d_0d^2_{\ns}\in\{1,2,3,6,49,98,147,294\}$.
Since there are no cusp forms of weight $2$ and level $14$ or less than $10$, we only consider $d_0d^2_{\ns}\in\{49,98,147,294\}$ corresponding to 15 Galois conjugacy classes of forms.
The set $K^\perp$ only contains $(1,1,1)$ and $(-1,-1,1)$, where $(k_1,k_2,k_3)$ corresponds to the values of the associated character on $(W_2,W_3,W_7)$. 
Since $\chi(W_7)=1$ for all $\chi\in K^\perp$, then we have $m_f = 0$ for each form $f$ having negative eigenvalue with respect to $W_7$ (i.e., having $\varepsilon_{f,7}=-1$).
Thus $6$ Galois conjugacy classes of forms are left to consider. For every $f$ of level $294$, we have $\varepsilon_{f,2}\chi(W_2)=\varepsilon_{f,3}\chi(W_3)=-1$, hence again $m_f=0$. Thus only 4 Galois conjugacy classes of forms $\{f_1,f_2,f_3,f_4\}$ need to be considered, one of level $98$ and three of level $147$, and each has $m_f=1$. 
The associated abelian varieties $A_f$ have dimension 2 except one which has dimension 1, so the genus is 7.
An alternative method to compute the genus is shown in \Cref{ese:genus_appendix}, in \Cref{sc:appendix_genus}.

For all forms $f_i$ and their Galois conjugates (seven forms in total) the Hecke eigenvalues $a_{11}(f)$ are equal to $-2$. Hence the Frobenius $\Frob_{11}$ acting on $X$ has two eigenvalues $\alpha_1,\alpha_2$, each of multiplicity $7$, which are the roots of the polynomial $x^2-a_{11}(f)x+11 = x^2+2x+11$. 
We deduce that the number of points of $X$ over $\FF_{11}$ is
\[
\#X(\FF_{11})=11+1-\sum_{i=1}^7 (\alpha_1+\alpha_2)=12+14=26 .
\]
Similarly the number of points over $\FF_{11^5}$ is 
\[
\#X(\FF_{11^5})=11^5+1-\sum_{i=1}^7 (\alpha_1^5+\alpha_2^5)=161052+5614=166666.
\]
This is our maximal curve of genus $7$, further studied in \Cref{sc:computing-equation}.
\end{ese}
\begin{ese}\label{ex_g=7}
We consider another quotient $X(6,7)$, namely
$X = X(6,7)/\langle W_3,W_7\rangle$. 
We follow the notation of Example \ref{ex_max}.
Again we only consider the 15 Galois conjugacy classes of level $d_0d^2_{\ns}\in\{49,98,147,294\}$ but now $K^\perp=\{(1,1,1),(-1,1,1)\}$.
Since $\chi(W_3)=\chi(W_7)=1$ for all $\chi\in K^\perp$, then we have $m_f = 0$ for each form $f$ having negative eigenvalue with respect to $W_3$ and $W_7$ (i.e., having $\varepsilon_{f,3}=-1$ or $\varepsilon_{f,7}=-1$).
Thus $3$ Galois conjugacy classes of forms are left to consider: $f_1,f_2,f_3$, of level $98$, $147$ and $294$, respectively. 
Among these $m_{f_1}=m_{f_3}=1$ and $m_{f_2}=2$.
The associated abelian varieties $A_f$ have dimension 2 except the last one which has dimension 1, so the genus is 7.

For the forms $f_1$ and $f_2$ and their Galois conjugates (six forms in total) the Hecke eigenvalues $a_{11}(f)$ are equal to $-2$ but, for the last form, $a_{11}(f)=5$. Hence the Frobenius $\Frob_{11}$ acting on $X$ has the same eigenvalues $\alpha_1$ and $\alpha_2$ of Example \ref{ex_max} with multiplicity $6$ and two new eigenvalues $\beta_1$ and $\beta_2$, which are the roots of the polynomial $x^2-a_{11}(f)x+11 = x^2-5x+11$. 
We deduce that the number of points of $X$ over $\FF_{11}$ is
\[
\#X(\FF_{11})=11+1-\beta_1-\beta_2-\sum_{i=1}^6 (\alpha_1+\alpha_2)=12-5+12=19 .
\]
Similarly the number of points over $\FF_{11^5}$ is 
\[
\#X(\FF_{11^5})=11^5+1-\beta_1^5-\beta_2^5-\sum_{i=1}^6 (\alpha_1^5+\alpha_2^5)=161052+725+4812=166589.
\]
This curve already belongs to the family studied in \cite{DLMS23} and the small computation above gives the same result contained there.
\end{ese}

We implemented the algorithm described in \cite{DLMS23} (and sketched in the examples) for curves of type $X(n_0,n_{\ns})/\langle W_{k_1},\ldots,W_{k_t}\rangle$, 
and applied it to all such curves with $n=n_0 n_{\ns}^2\leq 10000$. Following \cite{manypoints}, we also restricted our analysis to genus $g\le 50$ and field $\FF_{p^k}$ with $k\le 5$ and odd characteristic $p\leq 19$, respectively $k\le 7$ for $p=2$.

We got the eigenvalues of the Hecke operators on the relevant modular forms from \cite{lmfdb}: the database contains the $q$-expansions of all the newforms of level at most $10000$.

We found $36$ curves that, according to \cite{manypoints}, achieve a new (nice) record in terms of largest known number of points over $\FF_q$ given the genus $g$. Following \cite{manypoints}, a new record is considered \emph{nice} only when it is at least $q+1+\frac{M_g(\FF_q)-q-1}{\sqrt{2}}$, where $M_g(\FF_q)$ is the current (theoretical) upper bound for the number of points on a curve of genus $g$ over $\FF_q$ (in particular $M_q(\FF_q)$ is less than or equal to the Hasse-Weil-Serre bound).
We list these $36$ record curves in \Cref{tab:record1}.

\begin{center}
\begin{tabular}{c}
$
\begin{array}{|c|c|c|c|c||c|c|c|c|c|} 
\hline 
g & q & n_0, n_{\textnormal{ns}} & k_1,\ldots,k_t & \#X(\FF_q) & g & q & n_0, n_{\textnormal{ns}} & k_1,\ldots,k_t & \#X(\FF_q) \\
\hline  
 7  &  11^5  & 6,7 & 6,7  & 166666 & 35 &   17^2  &  208,3 & 16,39  & 1152\\
 10  &  17^5  & 208,1 & 208  & 1436834 & 36 &   5^2  &  1271,1 & 1271  & 238\\
 10  &  19^2  & 319,2 & 58,319  & 703 & 36 &   17^2  & 208,3 & 208,39  & 1182\\
 13  &  13^5  & 378,1 & 54,7  & 385275 & 38 &   2^2  &  29,7 & 203  & 65 \\
 14  &  7^2  & 110,3 & 3,55  & 196 & 38 &   5^2  &  593,2 & 1186  & 232 \\
 17  &  11^2  & 480,1 & 5,96  & 416  & 39 &   2^2  &  897,1 & 299  & 68\\
 19  &  2^2  & 545,1 & 545  & 40  & 39 &   17^2  &  104,3 & 39  & 1284\\
 22  &  3^2  & 1102,1 & 2,551  & 82 & 40 &   5^2  &  814,1 & 407  & 272\\
 24  &  11^2  & 306,1 & 34  & 530 & 40 &   19^2  &  1020,1 & 17,15  & 1440\\
 27  &  5^2  & 574,1 & 287  & 204 & 41 &   5^2  &  861,1 & 287  & 260\\
 28  &  5^2  & 23,7 & 161  & 202 & 42 &   3^2  &  1001,2 & 13,154  & 134\\
 29  &  5^2  & 1342,1 & 2,671  & 205 & 43 &   5^2  &  1343,1 & 1343  & 276\\
 30  &  5^2  & 1007,1 & 1007  & 216 & 44 &   2^2  &  1169,1 & 1169  & 70\\
 32  &  5^2  & 1246,1 & 2,623  & 214 & 45 &   5^2  &  696,1 & 87  & 280\\
 33  &  5^2  & 924,1 & 12,308  & 226 & 47 &   2^2  &  51,7 & 3,119  & 75\\
 34 &   3^2  &  638,1 & 319  & 114 & 48 &   5^2  &  2014,1 & 2,1007  & 306\\
 34 &   19^2  &  506,1 & 253  & 1300 & 49 &   2^2  &  1023,1 & 341  & 82\\
 \hline 
\end{array}
$
	\end{tabular}
	\captionof{table}{
Improved bounds for $\# X(\FF_q)$ 
among curves of genus $g$, here obtained with $X$ of the form $X(n_0,n_{\mathrm{ns}})/\langle W_{k_1},\ldots,W_{k_t} \rangle$ 
 with $n_0 n_{\textnormal{ns}}^2 \le 10000$ and $g\le 50$.
}
 \label{tab:record1}
\end{center}

\section{Explicit equations for two new maximal curves}\label{sc:computing-equation}

Among the curves we analyzed, we found two curves that are the first known curves, for their genera, to reach the Hasse-Serre-Weil upper bound over $\FF_{11^5}$. The two curves are $X(6,7)/\langle W_6,W_7\rangle$ and $X(156,1)/\langle W_{13}\rangle$: the first curve was unknown, to our knowledge, the second one already appeared in \cite{DLMS23}, but we have not computed explicit equations nor have we underlined the maximality. We give (singular) planar equations for both curves and we also give smooth equations for the first one, the one with lowest genus, using the canonical embedding.

\subsection{Canonical model of $\boldsymbol{X(6,7)/\langle W_6,W_7\rangle}$}
The following equations describe a canonical model of $X(6,7)/\langle W_6,W_7\rangle$ in $\PP^6$ over $\QQ$, which has genus $7$, it is smooth over $\ZZ[\tfrac{1}{42}]$ and reaches Serre's bound, as computed in \cref{tab:record1}.
\allowdisplaybreaks
{\small
\begin{align*}
& 10x_1^2 + 22x_2x_5 - 18x_2x_6 + 5x_2x_7 - 9x_3^2 - 14x_3x_4 - x_3x_5 - 4x_3x_6 + 7x_3x_7 - 4x_4^2 +\\
&\quad - 13x_4x_5 + 11x_4x_6 - 16x_4x_7 + 2x_5^2 - 17x_5x_6 - 7x_5x_7 + 10x_6^2 + x_6x_7 - x_7^2=0, \\
& -4x_1^2 + 28x_1x_2 - 56x_2x_5 + 49x_2x_6 + 17x_3^2 + 23x_3x_4 + 21x_3x_5 - 6x_3x_7 - 8x_4^2 +\\
&\quad - 2x_4x_5 - 2x_4x_6 - x_4x_7 - 16x_5^2 + 56x_5x_6 + 3x_5x_7 - 23x_6^2 + 13x_6x_7 - 7x_7^2=0, \\
& -6x_1^2 - 13x_1x_2 + x_1x_3 + 9x_2x_5 - 9x_2x_6 - 4x_2x_7 - x_3^2 - 9x_3x_5 + 3x_3x_6 - 3x_3x_7 + 7x_4^2 +\\
&\quad + 10x_4x_5 - 7x_4x_6 + 13x_4x_7 + 6x_5^2 - 13x_5x_6 + 4x_5x_7 + 3x_6^2 - 7x_6x_7 + 4x_7^2=0, \\
& -6x_1^2 - 21x_1x_2 + x_1x_4 + 15x_2x_5 - 18x_2x_6 - 9x_2x_7 - 3x_3^2 - 2x_3x_4 - 14x_3x_5 + 4x_3x_6 +\\
&\quad - 2x_3x_7 + 10x_4^2 + 16x_4x_5 - 10x_4x_6 + 15x_4x_7 + 12x_5^2 - 24x_5x_6 + 5x_5x_7 + 5x_6^2 - 13x_6x_7 + 6x_7^2=0, \\
& -3x_1^2 - 8x_1x_2 + x_1x_5 + 11x_2x_5 - 8x_2x_6 - x_2x_7 - 2x_3^2 - 2x_3x_4 - 6x_3x_5 + 2x_3x_6 + \\
&\quad - x_3x_7 + 4x_4^2 + 4x_4x_5 - 3x_4x_6 + 8x_4x_7 + 3x_5^2 - 11x_5x_6 + x_5x_7 + 4x_6^2 - 3x_6x_7 + 3x_7^2=0, \\
& -8x_1^2 - 2x_1x_2 + x_1x_6 - 13x_2x_5 + 12x_2x_6 - 2x_2x_7 + 6x_3^2 + 10x_3x_4 + 3x_3x_6 - 5x_3x_7 +\\
&\quad + 4x_4^2 + 9x_4x_5 - 8x_4x_6 + 13x_4x_7 - 2x_5^2 + 11x_5x_6 + 5x_5x_7 - 6x_6^2 + x_7^2=0, \\
& -6x_1^2 - 13x_1x_2 + x_1x_7 + 9x_2x_5 - 9x_2x_6 - 4x_2x_7 - x_3^2 - 9x_3x_5 + 3x_3x_6 - 3x_3x_7 + 7x_4^2 +\\
&\quad + 12x_4x_5 - 8x_4x_6 + 13x_4x_7 + 6x_5^2 - 13x_5x_6 + 4x_5x_7 + 3x_6^2 - 7x_6x_7 + 4x_7^2=0, \\
& -4x_1^2 - 7x_1x_2 + x_2^2 + 4x_2x_5 - 3x_2x_6 - x_2x_7 + x_3x_4 - 5x_3x_5 + 2x_3x_6 - 2x_3x_7 + 4x_4^2 +\\
&\quad + 7x_4x_5 - 5x_4x_6 + 8x_4x_7 + 2x_5^2 - 5x_5x_6 + 3x_5x_7 + x_6^2 - 3x_6x_7 + 2x_7^2=0, \\
& x_2x_3 - 6x_2x_5 + 3x_2x_6 - x_2x_7 + x_3^2 + 2x_3x_4 + x_3x_5 + 2x_4x_5 - x_4x_6 - 2x_4x_7 +\\
&\quad + 4x_5x_6 + x_5x_7 - 2x_6^2 - x_6x_7 - x_7^2=0, \\
& -4x_1^2 - 21x_1x_2 + x_2x_4 + 26x_2x_5 - 23x_2x_6 - 3x_2x_7 - 6x_3^2 - 7x_3x_4 - 15x_3x_5 + 3x_3x_6 +\\
&\quad - x_3x_7 + 9x_4^2 + 11x_4x_5 - 6x_4x_6 + 13x_4x_7 + 10x_5^2 - 29x_5x_6 + 3x_5x_7 + 10x_6^2 - 10x_6x_7 + 6x_7^2=0.
\end{align*}
}
\allowdisplaybreaks[0]
In \cite[file \texttt{\detokenize{Equations_curve_X(6,7)_mod_w6,w7.txt}}]{script_record_curves} we have written the above equations in Magma language, which we computed using the Magma script \texttt{\detokenize{Computation_of_equations.txt}} in the same repository.

In the remainder of this subsection we describe how we computed the equations above and prove the correctness of the computation.

Let $X = X(6,7)/\langle W_6, W_7\rangle$ be our maximal curve. 
To compute a canonical model it is enough to compute a basis of $\Omega^1(X)$, since the methods of \cite{Mer18} allow to find quadrics defining the image of a canonical embedding.
We have $X =  Y/\langle W_6\rangle$, where $Y = X(6,7)/\langle W_7\rangle$, and the idea is to compute $\Omega^1(X)$ as the subspace of $\Omega^1(Y)$ which is $W_6$-invariant. Since $Y$ is of the form $X_H$, the algorithm by David Zywina in \cite{Zyw21} (see also \cite[files \texttt{GL2GroupTheory.m} and \texttt{ModularCurves.m}]{Zywcode}) computes a basis of $S_2(\Gamma_H) = \Omega^1(Y)$. In our algorithm we numerically compute a candidate $S$ for $\Omega^1(X)$, then we compute the candidate equations (the ones above). Finally we prove, by geometric arguments, that the equations we found must be equations for $X$.

For the computation of $S$ we use that $W_6$ acts on the upper half-plane by the matrix $m_6 =  \left(\begin{smallmatrix}6&-1\\ 6&0\end{smallmatrix}\right)$. Given a basis $b_1, \ldots, b_{16}$ of $S_2(\Gamma_H) = \Omega^1(Y)$ we computed the values
\[
b_r(\tau_s) ,\quad (b_r|_2m_6)(\tau_s)
\]
with precision $10^{-10}$
for $\tau_s= \frac{s}{10g}+i\sqrt{\frac{1}{6}-\frac{s^2}{100g^2}}$, with $s=-15, \dots, 15$  (see \cite[Section 2.1, Equation 2.1.5]{Miy05} for the definition of $f|_2 m$ for $m\in\mathrm{GL}_2^+(\RR)$). 
By solving a linear system we computed $b_r|_2m_6$ as a linear combination of the $b_i$'s with approximated coefficient, denote it $(b_r|_2m_6)^\sim$. 
We know that $\Omega^1(X) = \Omega^1(Y)^{W_6}$ is generated by the forms $b_r+b_r|_2m_6$ and we define $S$, our numerical candidate for $\Omega^1(X)$, as the span of $b_r+(b_r|_2m_6)^\sim$; in the computation such forms $b_r+(b_r|_2m_6)^\sim$ turned out to be linearly dependent and we extracted a basis $s_1,\ldots,s_7$.
The values for $\tau$ are carefully chosen to have $\mathrm{Im}(\tau)$ as close as possible to $\mathrm{Im}(W_6\tau)$ and the coefficients of the linear combinations are rationals with small denominators (we look for denominators less than 40) up to a precision of $10^{-10}$ (using 5800 Fourier coefficients of the forms $b_r$). 

A first clue for $S$ being $\Omega^1(X)$ is that $S$ has dimension $7$. Using the method in \cite{Mer18} to compute canonical embeddings, we computed by \cite[file \texttt{\detokenize{Computation_of_equations.txt}}]{script_record_curves} quadratic relations which hold for a basis of forms $s_1, \ldots, s_7$ of $S$. Such quadratic relations are the equations above, which define a smooth curve over $\QQ$ of genus $7$ whose reduction modulo $11$ is smooth and has the expected number of points over $\FF_{11^5}$, which is the number of points of $X$ computed independently in \cref{sc:computing-points}, giving further clues. 

To turn clues into a proof, we notice that the differential forms $s_1, \ldots, s_7 \in \Omega^1(Y)$ define a map $Y \dashrightarrow X \subset \PP^6$, which is defined over $\QQ$ and has degree $2$ by Riemann-Hurwitz Formula (indeed $s_1, \ldots, s_7$, despite being chosen with an approximation process, were chosen in $\Omega^1(Y)$, and the equations above are quadratic relations that are exactly satisfied by the $s_i$'s). Hence $X$ is a quotient by an automorphism $u$ of $Y$ defined over $\QQ$. 
The automorphisms of a curve of genus at least $2$ are determined by their action on the Jacobian. Hence, it is enough to prove that $u = W_6$ in the Jacobian of $Y$. 
Applying \Cref{pr:chen_quotients_all}, we compute that the Jacobian $J_Y$ of $Y$ satisfies
$$J_Y \sim \prod_{i=1}^6 A_{h_i}^{m_i}$$ 
for certain distinct newforms $h_1, \ldots, h_6$ with multiplicities $m_i \in \{1,2\}$. 
Each $A_{h_i}$ is simple and non-isomorphic to the other factors. By \cite[Theorem 6]{Kani}, the ring of endomorphisms of $J_Y$ defined over $\QQ$, up to isogeny, is 
\begin{equation}\label{eq:EndJ_Y}
\End_\QQ^0(J_Y) = \End_\QQ(J_Y) \otimes \QQ \cong \prod_{i=1}^6 \End_\QQ^0(A_{h_i})^{m_i\otimes m_i} \cong  \prod_{i=1}^6 K_{h_i}^{m_i \times m_i} ,
\end{equation}
where $K_{h_i}$ is the field of coefficients of $h_i$ and in this concrete case is either $\QQ$ or a real quadratic extension of $\QQ$.
Using \eqref{eq:EndJ_Y}, we can write the action of $u$ on $J_Y$ as a collection of matrices $M_{u,i} \in K_{h_i}^{m_i \times m_i}$; analogously we have matrices $M_{W_2,i}$ and $M_{W_3,i}$ for the action of $W_2$ and $W_3$.

We now prove that $u$ commutes with $W_2$ or, equivalently, that for each $i$ the matrices $M_{u,i}$ and $M_{W_2,i}$ commute. 
Since $u$ and $W_2$ are involutions, the matrices $M_{u,i}$ and $M_{W_2,i}$ have order at most $2$, hence each matrix is either a multiple of the identity or a $2\times 2$ matrix with eigenvalues $\lambda_1=1$, $\lambda_2 = -1$. 
Suppose by contradiction that for some $i$ the matrices $M_{u,i}$ and $M_{W_2,i}$ do not commute. Then neither of them is a multiple of the identity, hence both of them have eigenvalues $1$ and $-1$. In particular $\det(M_{u,i}) = \det(M_{W_2,i})= -1$ and $\det(M_{u,i} M_{W_2,i}) = 1$. Now notice that the matrix $N_i := M_{u,i} M_{W_2,i}$ is a matrix of finite order, since, using \eqref{eq:EndJ_Y}, it describes the action of the automorphism $u W_2$ on a factor of the Jacobian of $Y$ and the group $\Aut(Y)$ is finite. Moreover, the order of $N_i$ is in $\{1,2,3,4,5,6,8,10,12\}$: indeed since $N_i$ is a $2\times 2$ matrix with entries in a quadratic extension of $\QQ$, its eigenvalues have degree over $\QQ$ dividing $4$; the only roots of unity with degree dividing $4$ are the ones whose order is in that list.
If $N_i$ has order $1$, then $M_{u,i}$ and $ M_{W_2,i}$ are inverse to each other, hence they commute. If $N_i$ has order $2$, using that $\det(N_i)=1$, we deduce that $N_i  = M_{u,i} M_{W_2,i} = - \Id$, again implying that $M_{u,i}$ and $ M_{W_2,i}$ commute. In the other cases we deduce that $u W_2$ has order a multiple of $3,4$ or $5$, contradicting the computations in \cite[file \texttt{\detokenize{Order_3_4_5_automorphisms_test.txt}}]{script_record_curves} that prove, using the test in \cite{Gon17}, that $Y$ has no automorphism of order $3$, $4$ or $5$.

Analogously $u$ commutes with $W_6$, i.e., with the whole $\langle W_2, W_6\rangle$.
Hence the elements of $\langle W_2, W_6\rangle$ descend to automorphisms of $X$ defined over $\QQ$.
If $u$ were not in $\langle W_2, W_6\rangle$, then all the elements of $\langle W_2, W_6\rangle$ would descend to distinct automorphisms of $X$ defined over $\QQ$, contradicting the Magma computations given by \cite[file \texttt{\detokenize{Check_non-modular_automorphisms_mod_5.txt}}]{script_record_curves}, which shows that $\# \Aut(X_{\FF_5})\le2$, hence $\# \Aut(X_{\QQ})\le2$ (indeed the composition $\Aut(X_\QQ) \to \Aut(X_{\FF_5}) \to \End(J_{\FF_5})$ is equal to $\Aut(X_\QQ) \to \End(J_{\QQ}) \to \End(J_{\FF_5})$ which is injective). Hence $u$ lies in the group $\langle W_2, W_6\rangle$. Furthermore, by looking at the genus of $Y/\langle W_2\rangle$ we exclude the case $u=W_2$ and by looking at the number of points of the quotient curves $Y/\langle W_3\rangle$ and $Y/\langle W_6\rangle$ over some finite field (computed in Examples \ref{ex_max} and \ref{ex_g=7}), we see that $u$ must be equal to $W_6$, i.e., that the above equations actually describe the curve $X = Y/\langle W_6\rangle$.

\subsection{A planar singular model for $\boldsymbol{X(6,7)/\langle W_6,W_7\rangle}$} \label{Sect4.1}

We can obtain a (singular) planar model of $X(6,7)/\langle W_6,W_7\rangle$ by projecting the canonical model onto a plane. In practice, we did it by a linear change of variables and eliminating all variables except the first three. Here is a planar equation for our curve, after dehomogenizing and reducing modulo~$11$:
\begin{align*}
& (9 x {-} x^{2}) y^{9}\! {+} (4 x^{3}\! {+} 4 x^{2}\! {+} 2 {-}1) y^{8}\! {+} (7 x^{4}\! {-} x^{3}\! {+} 4 x^{2}\! {+} 3) y^{7}\! {+} (5 x^{4}\! {+} 8 x^{3}\! {+} x^{2}\! {+} 2 x {+} 1) y^{6}\! {+} (8 x^{6}\! {+} 3 x^{5}\! {+} x^{4}\! {+} 9 x^{3}\! {+} 3 x {+} 5) y^{5}\! \\
&  \,\, {+} (9 x^{7}\! {+} 7 x^{6}\! {+} 2 x^{5}\! {+} x^{4}\! {+} 9 x^{3}\! {+} 4 x {+} 6) y^{4}\! {+} (8 x^{8}\! {+} 6 x^{7}\! {+} 8 x^{5}\! {-} x^{4}\! {+} 6 x^{3}\! {+} 8 x^{2}\! {+} 7 x {+} 3) y^{3}{+} (3 x^{8}\! {+} 8 x^{6}\! {+} 3 x^{5}\! {+} 5 x^{4}\! {+} 2 x^{3}\! {-} x^{2}\!  \\
& \,\, {+} x {+} 1) y^{2}\! {+} (2 x^{10\!} {+} 9 x^{9}\! {+} 9 x^{8}\! {+} 5 x^{7}\! {+} 6 x^{6}\! {+} 7 x^{5}\! {+} 6 x^{4}\! {+} 5 x^{2}\! {+} 6 x) y  \! {+} 7 x^{11}\! {+} 5 x^{10}\! {+} 2 x^{9}\! {-} x^{8}\! {+} 6 x^{7}\! {-} x^{6}\! {+} 7 x^{5}\! {+} 4 x^{3} = 0.
\end{align*}
As it can be computed using the dedicated Magma command, the above equation defines a curve whose normalization of the projective closure, which we denote $\widetilde X$, has genus $7$.  
It is not so hard to (re)prove, with the help of Magma, that $\widetilde X$  has $166666$ points over $\FF_{11^5}$: the closure in $\PP^2_{\FF_{11^5}}$ of the above planar model has $166665$ rational points of which three are singular: $(0:0:1)$, $(0:1:0)$ and $(9:7:1)$; the normalization $\widetilde X$ contains $4$ rational points lying over these $3$ singularities since all of them are nodal, the first two are split while the third is not.

The above equation can be found at \cite[file \texttt{\detokenize{Equations_curve_X(6,7)_mod_w6,w7.txt}}]{script_record_curves} in Magma readable format.

\subsection{Equations for $\boldsymbol{X(156,1)/\langle W_{13}\rangle}$.} 

In the family of curves we considered, there is another maximal curve, of genus 12, namely $X(156,1)/\langle W_{13}\rangle$, which has $170676$  over $\mathbb F_{11^5}$.
We also computed equations for its canonical embedding which can be found in Magma readable format at \cite[file \texttt{\detokenize{Equations_curve_X0(156)_mod_w13.txt}}]{script_record_curves}.

This curve is actually the same as $X_0(156)/\langle W_{13}\rangle$, thus to compute equations we used classical methods as those in \cite{Gal99} and \cite{Mer18}: 
\begin{itemize}
    \item We computed the $q$-expansion of a basis of eigenforms $\{f_1,\dots,f_g\}$ of $\Omega^1(X_0(156))$ up to $q^m$, with $m=180$.
    \item Using Magma's exact computation of the Atkin-Lehner eigenvalues, we computed a basis $g_1, \ldots, g_{12}$ of the space of invariants $\Omega^1(X_0(156))^{W_{13}} = \Omega^1(X)$ (again we computed the $q$-expansion of these forms up to $q^{180}$).
    \item We looked for quadratic polynomials $F_j$ such that $F_j(g_1, \ldots, g_{12}) = O(q^{180})$: since $F_j(g_1, \ldots, g_{12})$ is a section of the sheaf $(\Omega^1_X)^{\otimes2}$ which has degree $4g-4 = 44$, then $F_j(g_1, \ldots, g_{12})$ must be exactly $0$, otherwise it would have more than $44$ zeroes, which is absurd (see also \cite[Section 2.1, Lemma 2.2, p. 1329]{BGGP05}).
    \item After collecting enough $F_j$'s we realized that they cut a curve $C \subset \PP^{11}$ of genus $12$ and that this curve $C$ is isomorphic to $X$: indeed the image of the canonical map $X \to \PP^{11}$ given by the sections $g_i$ is contained in $C$ and, by Riemann-Hurwitz, the induced map $X \to C$ is an isomorphism.  
\end{itemize}

A planar singular model for the curve, when reduced modulo~$11$, is given by the following equation:
\newline
\newline
\parbox[b]{0.99\textwidth}{\raggedright\hangafter=1\hangindent=0.5em$\displaystyle
\, x^{19}\left( 6  {+}\allowbreak  2 y\right) +\allowbreak  x^{18}\left({-} 1  {+}\allowbreak  5 y  {+}\allowbreak  5 y^{2} \right)+ x^{17}\left(8  {-} y  {+}\allowbreak  7 y^{2}  {+}\allowbreak  y^{3} \right) +\allowbreak  x^{16}\left({-} y^{2}  {+}\allowbreak  4 y^{3}  {+}\allowbreak  y^{4} \right) +\allowbreak 
x^{15}\left(3 y  {+}\allowbreak  2 y^{2}  {+}\allowbreak  y^{5} \right) +\allowbreak 
x^{14}\left(4  {+}\allowbreak  6 y  {+}\allowbreak  2 y^{2}  {+}\allowbreak  9 y^{4}  {-} y^{5}  {+}\allowbreak  5 y^{6} \right) +\allowbreak 
x^{13}\left(2  {+}\allowbreak  y  {+}\allowbreak  8 y^{2}  {+}\allowbreak  2 y^{3}  {+}\allowbreak  y^{4}  {+}\allowbreak  3 y^{5}  {+}\allowbreak  y^{6}  {+}\allowbreak  3 y^{7} \right) +\allowbreak  x^{12}\big( 8  {+}\allowbreak  4 y  {+}\allowbreak  9 y^{2}  {+}\allowbreak  4 y^{3}  {+}\allowbreak  2 y^{4}
{+}\allowbreak  5 y^{5}  {+}\allowbreak  y^{6}  {+}\allowbreak  9 y^{7}  {-} y^{8} \big) +\allowbreak  x^{11}\left( 9  {-} y  {+}\allowbreak  3 y^{2}  {-} y^{3}  {+}\allowbreak  3 y^{4}  {+}\allowbreak  y^{5}  {+}\allowbreak  6 y^{6}  {+}\allowbreak  8 y^{7}  {+}\allowbreak  9 y^{8}  {+}\allowbreak  4 y^{9} \right) +\allowbreak  x^{10} \big( 6  {-} y  {+}\allowbreak  6 y^{2}  {-} y^{3}  {+}\allowbreak  6 y^{5}  {+}\allowbreak  
5 y^{6}  {-} y^{7}  {+}\allowbreak  6 y^{8}  {+}\allowbreak  6 y^{9}  {+}\allowbreak  4 y^{10} \big) +\allowbreak  x^9 \big( 9  {+}\allowbreak  4 y  {+}\allowbreak  8 y^{2}  {+}\allowbreak  y^{3}  {-} y^{4}  {+}\allowbreak  8 y^{5}  {+}\allowbreak  9 y^{6}  {+}\allowbreak  4 y^{7}  {+}\allowbreak  8 y^{8}  {+}\allowbreak  6 y^{10}  \big) +\allowbreak  x^8 \big( 2  {+}\allowbreak  3 y  {+}\allowbreak  9 y^{2}  {+}\allowbreak  2 y^{3}
{+}\allowbreak  4 y^{4}  {+}\allowbreak  6 y^{5}  {+}\allowbreak  2 y^{6}  {+}\allowbreak  4 y^{7}  {+}\allowbreak  9 y^{8}  {+}\allowbreak  4 y^{9}  {+}\allowbreak  6 y^{10}  {+}\allowbreak  7 y^{11}  {+}\allowbreak  4 y^{12} \big) +\allowbreak  x^7 \big( 5  {+}\allowbreak  8 y  {+}\allowbreak  6 y^{2}  {+}\allowbreak  8 y^{3}  {+}\allowbreak  6 y^{4}  {-} y^{5}  {+}\allowbreak  8 y^{6}  {+}\allowbreak  y^{7}  {+}\allowbreak  4 y^{8}  {+}\allowbreak  8 y^{9}  {+}\allowbreak 
    2 y^{10}  {+}\allowbreak  y^{11}  {+}\allowbreak  7 y^{13} \big) +\allowbreak 
x^6 \big(  1 {+}\allowbreak  9 y  {+}\allowbreak  3 y^{2}  {+}\allowbreak  4 y^{3}  {+}\allowbreak  3 y^{4}  {+}\allowbreak  y^{5}  {+}\allowbreak  6 y^{6}  {+}\allowbreak  4 y^{7}  {+}\allowbreak  7 y^{8}  {+}\allowbreak  7 y^{9}  {+}\allowbreak  4 y^{10}  {+}\allowbreak  2 y^{11}  {+}\allowbreak  4 y^{12}  {+}\allowbreak  2 y^{13}  {+}\allowbreak  5 y^{14} \big) +\allowbreak 
x^5 \big( 4  {+}\allowbreak  7 y  {+}\allowbreak  5 y^{2}  {+}\allowbreak  9 y^{3}  {-} y^{4}  {+}\allowbreak  2 y^{5}  {+}\allowbreak  2 y^{6}  {+}\allowbreak  3 y^{7}  {+}\allowbreak  3 y^{8}  {+}\allowbreak  y^{9}  {+}\allowbreak  5 y^{10}  {+}\allowbreak  8 y^{12}  {+}\allowbreak  9 y^{13}  {+}\allowbreak  3 y^{14}  {+}\allowbreak  2 y^{15} \big) +\allowbreak 
x^4 \big( 4  {-} y  {+}\allowbreak  3 y^{2}  {+}\allowbreak  5 y^{3}  {+}\allowbreak  2 y^{4}  {+}\allowbreak  6 y^{5}  {+}\allowbreak  7 y^{7}  {+}\allowbreak  y^{8}  {+}\allowbreak  5 y^{12}  {+}\allowbreak  5 y^{14}  {+}\allowbreak  7 y^{15}  {+}\allowbreak  9 y^{16} \big) {+}\allowbreak 
x^3 \big( 9  {-} y  {+}\allowbreak  9 y^{2}  {+}\allowbreak  6 y^{3}  {+}\allowbreak  6 y^{4}  {-} y^{6}  {+}\allowbreak  2 y^{7}  {+}\allowbreak  7 y^{8}  {+}\allowbreak  8 y^{9}  {+}\allowbreak  4 y^{10}  {+}\allowbreak  y^{11}  {+}\allowbreak  7 y^{12}  {+}\allowbreak  7 y^{13}  {+}\allowbreak  y^{14}  {+}\allowbreak  y^{15}  {+}\allowbreak  3 y^{16}  {+}\allowbreak  3 y^{17} \big) +\allowbreak 
x^2 \big( 4  {+}\allowbreak  9 y^{2}  {+}\allowbreak  y^{3}  {+}\allowbreak  6 y^{4}  {-} y^{5}  {+}\allowbreak  5 y^{6}  {+}\allowbreak  5 y^{7}  {-} y^{8}  {+}\allowbreak  y^{9}  {-} y^{10}  {+}\allowbreak  2 y^{11}  {+}\allowbreak  3 y^{13}  {+}\allowbreak  2 y^{15}  {+}\allowbreak  9 y^{16}  {+}\allowbreak  9 y^{17}  {+}\allowbreak  8 y^{18} \big) +\allowbreak 
x \big(3   {+}\allowbreak  7  y  {-}  y^{2}  {+}\allowbreak  3  y^{3}  {+}\allowbreak  4  y^{4}  {+}\allowbreak  9  y^{5}  {-}  y^{6}  {+}\allowbreak  8  y^{7}  {+}\allowbreak  3  y^{8}  {+}\allowbreak  3  y^{9}  {+}\allowbreak   y^{10}  {+}\allowbreak  8  y^{11}  {+}\allowbreak  4  y^{12}  {+}\allowbreak  5  y^{13}  {+}\allowbreak  3  y^{16}  {+}\allowbreak  9  y^{17}  {+}\allowbreak  3  y^{18}  {+}\allowbreak  6  y^{19}\big) +\allowbreak 
5  {+}\allowbreak  y  {+}\allowbreak  6 y^{2}  {+}\allowbreak  6 y^{3}  {+}\allowbreak  9 y^{4}  {+}\allowbreak  2 y^{6}  {-} y^{7}  {+}\allowbreak  2 y^{8}  {+}\allowbreak  3 y^{9}  {+}\allowbreak  7 y^{10}  {+}\allowbreak  7 y^{11}  {+}\allowbreak  3 y^{12}  {+}\allowbreak  3 y^{13}  {+}\allowbreak  6 y^{14}  {+}\allowbreak  9 y^{15}  {+}\allowbreak  3 y^{16}  {+}\allowbreak  2 y^{17}  {+}\allowbreak  3 y^{19}  {+}\allowbreak  y^{20} =0.
$}
\newline
\newline
It has been obtained from the canonical model by projection onto a plane, i.e., in practice by a linear change of variable plus variable elimination plus dehomogenization. The previous equation in Magma readable format can be found at \cite[file \texttt{\detokenize{Equations_curve_X0(156)_mod_w13.txt}}]{script_record_curves}.

\section{Other maximal quotient modular curves}\label{sc:other}

Besides the two curves analyzed in \Cref{sc:computing-equation}, inside the set of curves studied in \Cref{sc:computing-points} we found several other examples of maximal curves. We did not compute explicit equations for all these examples, since, as tracked by \cite{manypoints}, a maximal curve for that genus and field is already known. 
We collected these instances for genus larger than 1, and we computed the eigenvalues of Frobenius and the real Weil polynomial. The data is shown in Appendix~\ref{app:maxcurves}.

For some pairs (genus, finite field), we find many modular curves that are maximal and, in some cases, they have different real Weil polynomial. In particular this shows the existence of non-isomorphic maximal curves. 
This problem has also been studied by Rigato: in \cite{Rig10} she proved that for $q=2$ and $g\leq 5$ the maximal curves are all isomorphic over $\FF_2$ and she shows the existence of non-isomorphic maximal curves of genus 6 and 7.
We also point out the that all the maximal curves of genus 4 over $\FF_{2^7}$ are isomorphic by the 
uniqueness remark made by Everett Howe in \cite{manypoints} (see \cite{MPIA} for a static version). In particular all curves in the corresponding line of Table \ref{tabMax34} are isomorphic over $\FF_{2^7}$.

We now make more explicit a case of genus 4 where we observe more than one isomorphism class of maximal curves, and one case of genus 6 where explicit equations can be given.
\begin{ese}
In genus $4$ over the field $\FF_{2^2}$ we observe at least four isomorphism classes of maximal curves, among the curves we analyzed (see \Cref{tab:record1}). The following three curves have different real Weil polynomial (as in \cref{tabMax34}):
\begin{align*}
X(1, 11) =X_{\textnormal{ns}}(11): &\begin{cases}
x^{2} + x y + x z + x w + y^{2} + y z + y w + z^{2} + z w=0 \\
x^{2} z + x^{2} w + x z^{2} + x w^{2} + y^{3} + y z w + y w^{2} + z^{3} + z w^{2} + w^{3}=0,
\end{cases}\\
 X(81, 1)=X_0(81): &\begin{cases}
xw + y^2=0 \\
x^3 + y^3 + z^3 + z^2w + zw^2=0,
\end{cases}\\
X(251, 1)/\langle W_{251}\rangle = X_0^+(251): &\begin{cases}
x^2 + xy + yz + yw + zw=0 \\
x^2z + xz^2 + y^2w + yw^2=0.
\end{cases} 
\end{align*}
The equations of $X_{\textnormal{ns}}(11)$ can be found in \cite{lmfdb}, while equations (over $\QQ$) for the other curves were computed with the code \cite[file \texttt{\detokenize{Equations_other_maximal_curves.txt}}]{script_record_curves}.
Although we only found three distinct real Weil polynomials, the isomorphism classes are at least 4, in fact one can check with Magma that
\[
X(177, 1)/\langle W_{59}\rangle = X_0(177)/\langle W_{59}\rangle:\begin{cases}
x^2 + xy + y^2 + yz + z^2 + zw=0 \\
x^2z + xz^2 + xzw + y^2w + yzw + yw^2 + z^2w=0
\end{cases}
\]
is not isomorphic to $X_0(81)$ over $\FF_{2^2}$ even if they have the same real Weil polynomial.
\end{ese}

\begin{ese}
In genus $6$ and over the field $\FF_{2^2}$, we find only one curve. Since there are no explicit equations on website \cite{manypoints}, we display below some equations for the curve we found, computed with the code in \cite[file \texttt{\detokenize{Equations_other_maximal_curves.txt}}]{script_record_curves}.
\[
X(299, 1)/\langle W_{299}\rangle=X_0^+(299) :\begin{cases}
x_1^2 + x_2^2 + x_2x_4 + x_2x_6 + x_3x_4 + x_5x_6=0 \\
x_1x_2 + x_2^2 + x_2x_3 + x_2x_6 + x_3^2 + x_3x_5 + x_4^2 + x_5x_6=0 \\
x_1x_3 + x_2x_6 + x_5x_6=0 \\
x_1x_4 + x_2x_3 + x_3x_5 + x_3x_6 + x_4^2=0 \\
x_1x_5 + x_2x_3 + x_3^2 + x_3x_6 + x_4x_6 + x_5x_6=0 \\
x_1x_6 + x_2x_5 + x_3x_4 + x_3x_5 + x_3x_6=0.
\end{cases}
\]
\end{ese}

\appendix
\setcounter{equation}{0}
\renewcommand{\theequation}{\thesection.\arabic{equation}}

\section{Jacobians of quotient curves}\label{sc:appendix_genus}

The Atkin-Lehner automorphisms on a curve $X_0(n)$ and, slightly more in general, the automorphisms that we consider for a curve $X(n_0,n_\ns)$, form a group isomorphic to $(\ZZ/2\ZZ)^r$. 
In the general setting where $X$ is a smooth projective curve over a field with an action of $(\ZZ/2\ZZ)^r$, this appendix gives some elementary results regarding the Jacobians of the quotient curves $X/W$ for $W$ a subgroup of $(\ZZ/2\ZZ)^r$.

In this setting, all the possible curves $X/W$ form a lattice, for example when $r=3$ we get the following, where $\{b_1,b_2,b_3\}$ is a basis of $(\ZZ/2\ZZ)^3$.
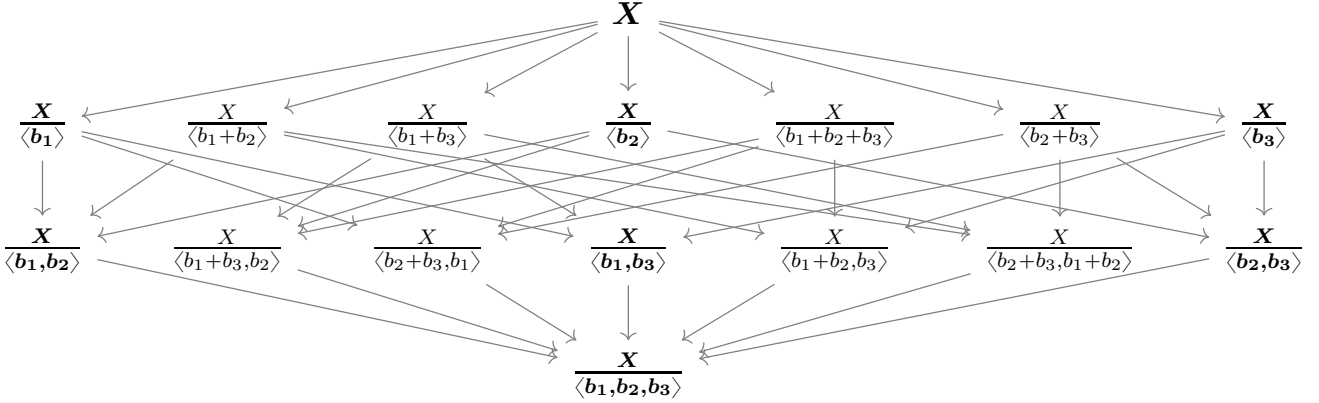
\[
\begin{tikzcd}[column sep=0.8cm, every arrow/.append style={draw=gray}]
    & & & \boldsymbol{X} \arrow[dl] \arrow[dll] \arrow[d] \arrow[dr] \arrow[drr] \arrow[drrr] \arrow[dlll] & & & & & \\ 
    \boldsymbol{\frac{X}{\langle b_1 \rangle}} \arrow[d] \arrow[drr] \arrow[drrr] & 
    \frac{X}{\langle b_1 + b_2 \rangle} \arrow[dl] \arrow[drrr] \arrow[drrrr]& 
    \frac{X}{\langle b_1 + b_3 \rangle} \arrow[dl] \arrow[dr] \arrow[drrr] & 
    \boldsymbol{\frac{X}{\langle b_2 \rangle}} \arrow[dlll] \arrow[dll] \arrow[drrr]& 
    \frac{X}{\langle b_1 + b_2 + b_3 \rangle}  \arrow[dll] \arrow[dlll] \arrow[d]& 
    \frac{X}{\langle b_2 + b_3 \rangle}  \arrow[d] \arrow[dr] \arrow[dlll] & 
    \boldsymbol{\frac{X}{\langle b_3 \rangle}} \arrow[d] \arrow[dll] \arrow[dlll]& 
    \\ 
    \boldsymbol{\frac{X}{\langle b_1, b_2 \rangle}} \arrow[drrr] & 
    \frac{X}{\langle b_1 + b_3, b_2 \rangle} \arrow[drr] & 
    \frac{X}{\langle b_2 + b_3, b_1 \rangle} \arrow[dr] & 
    \boldsymbol{\frac{X}{\langle b_1, b_3 \rangle}} \arrow[d] & 
    \frac{X}{\langle b_1 + b_2, b_3 \rangle} \arrow[dl] & 
    \frac{X}{\langle b_2 + b_3, b_1 + b_2 \rangle} \arrow[dll] &
    \boldsymbol{\frac{X}{\langle b_2, b_3 \rangle}} \arrow[dlll] 
    \\ 
    & & & \boldsymbol{\frac{X}{\langle b_1,b_2, b_3 \rangle}} & & & & 
\end{tikzcd}
\]

Given $W$ a subgroup of $(\ZZ/2\ZZ)^r$, we denote 
\begin{equation}\label{eq:defJH}
    J_W := \mathrm{Jac}(X/W).
\end{equation}
In \Cref{th:jac_formula} we prove that a few Jacobians $J_{W_i}$ (in the above diagram they correspond to the bold curves) determine all the other $J_W$. 
In the case $r=2$ we have the following statement.
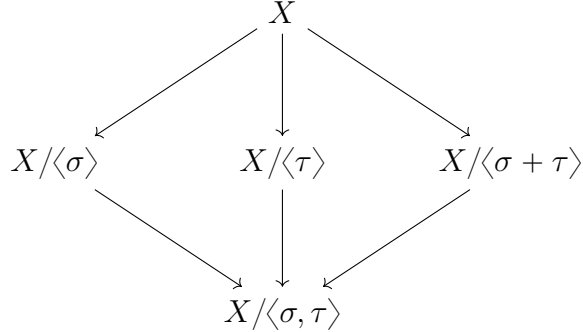
\begin{equation*} \label{eq:lattice_quot}  
\begin{tikzpicture}
[x=2cm,y=2cm]
\node at (0,1)   (A1)  {$X$};
\node at (-1.5,0)  (A2)  {$X/\langle \sigma \rangle$};
\node at (0,0)   (A3)  {$X/\langle \tau \rangle$};
\node at (1.5,0)   (A4)  {$X/\langle \sigma+\tau\rangle$};
\node at (0,-1)  (A5)  {$X/\langle \sigma,\tau\rangle$};
\draw[->] (A1) -- (A2);
\draw[->] (A1) -- (A3);
\draw[->] (A1) -- (A4);
\draw[->] (A2) -- (A5);
\draw[->] (A3) -- (A5);
\draw[->] (A4) -- (A5);
\end{tikzpicture}
\end{equation*}
\begin{lem}\label{lem:base_case}
Let $X$ be a smooth projective curve which is acted upon by the group $(\ZZ/2\ZZ)^2$, generated by elements $\sigma$ and $\tau$.
Denote by $J$ the Jacobian of $X$ and let $J_W$ be as in \eqref{eq:defJH}. Then, there is an isogeny 
\begin{align*}
J \times J_{\langle \sigma, \tau \rangle} ^2 \sim J_{\langle \sigma \rangle} \times J_{\langle \tau \rangle} \times J_{\langle \sigma + \tau \rangle}.
\end{align*}
\end{lem}
\begin{proof}
Given two signs $\epsilon_\sigma, \epsilon_\tau \in \{+, -\}$, denote by $J^{\epsilon_\sigma \, \epsilon_\tau}$ the connected component of the identity of the subgroup scheme  $\ker(\sigma^*-\epsilon_\sigma \Id ) \cap \ker(\tau^*-\epsilon_\tau \Id)$ in~$J$. In other words, $J^{\epsilon_\sigma\, \epsilon_\tau}$ is the maximal abelian subvariety of $J$ where $\sigma$ acts as $\epsilon_\sigma \Id$ and $\tau$ acts as $\epsilon_\tau \Id$. The inclusions $J^{\epsilon_\sigma\, \epsilon_\tau} \subset J$ induce an isogeny
\begin{equation} \label{eq:Jdecomposition}
J\sim J^{++} \times J^{+-} \times J^{-+} \times J^{--} \,.
\end{equation}
Finally recall that the Jacobian of $X/W$ is isogenous, by pullback of divisors, to the connected component of the identity of the $W$-invariant part of $J$. Hence our statement follows from \Cref{eq:Jdecomposition} together with the identities
\begin{align*}
J_{\langle \sigma, \tau \rangle} \sim J^{++}, \qquad 
J_{\langle \sigma \rangle} \sim J^{++} \times J^{+-}, \qquad 
J_{\langle \tau \rangle} \sim J^{++} \times J^{-+}, 
\qquad J_{\langle \sigma + \tau \rangle} \sim J^{+-} \times J^{-+}.
\end{align*}
\end{proof}
The following lemma generalizes \Cref{lem:base_case}. 
Given a subset $I$ of $(\ZZ/2\ZZ)^r$, we denote by $J_I = J_{\langle I\rangle}$ the Jacobian of $X/\langle I \rangle$, with $\langle I \rangle$ the subgroup generated by $I$. Moreover, the statement of the lemma contains a product $\prod A_i^{e_i}$ of abelian varieties $A_i$ with possibly negative exponents $e_i$: we interpret this product as a formal object living in the Grothendieck group of the category of abelian varieties up to isogeny; \emph{a posteriori}, the lemma itself implies that the product below is a real abelian variety, i.e., after decomposing every factor in irreducible abelian varieties, every irreducible factor appearing with a negative exponents cancels out.
\begin{lem}\label{lem:J_ind_step}
Let $X$ be a smooth projective curve with an action of $(\ZZ/2\ZZ)^r$ and let $\{b_1,\ldots,b_r\}$ be a basis of $(\ZZ/2\ZZ)^r$. 
For all subsets $K \subset (\ZZ/2\ZZ)^r$ and all non-negative integers $s\le r$, we have the following isogeny: 
\begin{align*}
J_{ K \cup \{b_1+\cdots+b_s\}}\sim J_{K}^{\tfrac{1+(-1)^{s}}{2}} \times \prod_{j=1}^s \, \Bigg(  \prod_{\substack{I\subseteq \{b_1,\ldots,b_s\}\\\# I =j}}  J_{I\cup K} \Bigg)^{(-1)^{j+s}\cdot2^{j-1}}.
\end{align*}
\end{lem}
\begin{proof}
To simplify formulas, we use the additive notation for the product of abelian varieties. 
The proof is by induction: the case $s=0$ is immediate, so we assume that the equality holds for $s-1$. By \Cref{lem:base_case} applied to the curve $X/\langle K \rangle$, with $\sigma = b_1+\cdots+b_{s-1}$ and $\tau = b_s$, we have
\begin{align*}
J_{\{ b_1+\cdots+b_s\} \cup K}&\sim J_{ K }+2J_{\{ b_1+\cdots+b_{s-1},b_s\} \cup K }-J_{\{ b_1+\cdots+b_{s-1}\} \cup K}-J_{\{ b_s\} \cup K}.
\end{align*}
By inductive hypothesis on the second and the third summand, we obtain
\allowdisplaybreaks
\begin{align*}
&J_{\{ b_1+\cdots+b_s\} \cup K}\sim  J_{K}+2\left(\tfrac{1-(-1)^{s}}{2}J_{\{b_s\} \cup K}+(-1)^{s}\sum_{j=1}^{s-1} (-2)^{j-1} \!\!\!\!\!\!\!\!\sum_{\substack{I\subseteq \{b_1,\ldots,b_{s-1}\}\\\#I=j}}\!\!\!\!\!\!\!\! J_{I\cup K\cup\{b_s\}}\right)+\\
& \qquad \qquad \qquad \qquad\qquad -\left(\tfrac{1-(-1)^{s}}{2}J_{K}+(-1)^{s}\sum_{j=1}^{s-1} (-2)^{j-1} \!\!\!\!\!\!\!\!\sum_{\substack{I\subseteq \{b_1,\ldots,b_{s-1}\}\\\#I=j}}\!\!\!\!\!\!\!\! J_{I\cup K}\right)-J_{\{b_s\} \cup K} 
\\
& \sim
\tfrac{1+(-1)^{s}}{2}J_{K}
+(-1)^{s+1}\sum_{j=1}^{s-1} (-2)^{j-1} \!\!\!\!\!\!\!\!\sum_{\substack{I\subseteq \{b_1,\ldots,b_{s-1}\}\\\#I=j}}\!\!\!\!\!\!\!\! J_{I\cup K} 
+(-1)^{s+1}\sum_{j=1}^{s-1} (-2)^{j} \!\!\!\!\!\!\!\!\sum_{\substack{I\subseteq \{b_1,\ldots,b_{s-1}\}\\\#I=j}}\!\!\!\!\!\!\!\! J_{I\cup K\cup\{b_s\}} 
+
(-1)^{s+1}J_{\{b_s\} \cup K}
\\
&\sim
\tfrac{1+(-1)^{s}}{2}J_{K}+(-1)^{s+1}
\sum_{j=1}^{s-1} (-2)^{j-1} \!\!\!\!\!\!\!\!\sum_{\substack{I\subseteq \{b_1,\ldots,b_{s}\}\\\#I=j\text{ and }b_s\notin I}}\!\!\!\!\!\!\!\! J_{I\cup K} 
+ (-1)^{s+1}\sum_{j=2}^{s} (-2)^{j} \!\!\!\!\!\!\!\!\sum_{\substack{I\subseteq \{b_1,\ldots,b_{s-1}\}\\\#I=j-1}}\!\!\!\!\!\!\!\! J_{I\cup K\cup\{b_s\}} + (-1)^{s+1}J_{\{b_s\} \cup K}
\\
&\sim
\tfrac{1+(-1)^{s}}{2}J_{K}
+(-1)^{s+1}\sum_{j=1}^{s} (-2)^{j-1} \!\!\!\!\!\!\!\!\sum_{\substack{I\subseteq \{b_1,\ldots,b_{s}\}\\\#I=j\text{ and }b_s\notin I}}\!\!\!\!\!\!\!\! J_{I\cup K}
+ (-1)^{s+1}\sum_{j=1}^{s} (-2)^{j-1} \!\!\!\!\!\!\!\!\sum_{\substack{I\subseteq \{b_1,\ldots,b_{s}\}\\\#I=j\text{ and }b_s\in I}}\!\!\!\!\!\!\!\! J_{I\cup K} 
\\
&\sim\tfrac{1+(-1)^{s}}{2}J_{K}+(-1)^{s+1}\sum_{j=1}^s (-2)^{j-1} \!\!\!\!\!\!\!\!\sum_{\substack{I\subseteq \{b_1,\ldots,b_s\}\\\#I=j}}\!\!\!\!\!\!\!\! J_{I\cup K}.
\end{align*}
\end{proof}
\begin{thm}\label{th:jac_formula}
Let $X$ be a smooth projective curve with an action of $(\ZZ/2\ZZ)^r$ and let $B$ be a basis of $(\ZZ/2\ZZ)^r$. 
For all subgroups $W$ of $(\ZZ/2\ZZ)^r$, the Jacobian $J_W$ is isogenous to an integer combination of the Jacobians $J_{\langle S\rangle}$, with $S$ varying among the subsets of $B$.
\end{thm}
\begin{proof}
Choose a set $I$ of generator for $W$ and apply \Cref{lem:J_ind_step} repeatedly to each element of $I$, which can be written as sum of elements of $B$.
\end{proof}
The above results imply analogous ``numerical'' corollaries. 
The next corollary, about genera, can also be proven directly, controlling the ramification of the quotient maps and applying Riemann-Hurwitz Formula.

\begin{cor}\label{cor:genera}
Let $X$ be a smooth projective curve with an action by $(\ZZ/2\ZZ)^r$ and let $B=\{b_1,\ldots,b_r\}$ be a basis of $(\ZZ/2\ZZ)^r$. For $I$ a subset of $(\ZZ/2\ZZ)^r$, we denote by $g_I$ the genus of $X/\langle I \rangle$.

For all subsets $K$  of $(\ZZ/2\ZZ)^r$ and all non-negative integers $s\le r$, we have
\begin{align*}
g_{\{b_1+\cdots+b_s\} \cup K}= \tfrac{1+(-1)^{s}}{2}g_{K}+(-1)^{s+1}\sum_{j=1}^s (-2)^{j-1} \!\!\!\!\!\!\!\!\sum_{\substack{I\subseteq \{b_1,\ldots,b_s\}\\\#I=j}}\!\!\!\!\!\!\!\! g_{I\cup K }.
\end{align*}
Moreover, for each subset $I$ of $(\ZZ/2\ZZ)^r$, the genus $g_I$ can be written as an integral linear combination of the genera $g_S$ with $S$ varying among all the subsets of $B$.
\end{cor}
\begin{proof}
It follows by \Cref{lem:J_ind_step} and \Cref{th:jac_formula} recalling that the genus of a curve is the dimension of the Jacobian as an abelian variety.
\end{proof}
\begin{cor}\label{cor:overFF}
Let $X$ be a smooth projective curve over a finite field $\FF_q$, with an action of $(\ZZ/2\ZZ)^r$ and let $B=\{b_1,\ldots,b_r\}$ be a basis of $(\ZZ/2\ZZ)^r$. 

For all subsets $K$ of $(\ZZ/2\ZZ)^r$ and all non-negative integer $s\le r$, we have:
\begin{align*}
&\#(X/\langle \{b_1+\cdots+b_s \} {\cup} K \rangle)(\FF_q)= \tfrac{1+(-1)^{s}}{2}\#(X/\langle K \rangle)(\FF_q)+(-1)^{s+1}\!\sum_{j=1}^s (-2)^{j-1} \!\!\!\!\!\!\!\!\!\sum_{\substack{I\subseteq \{b_1,\ldots,b_s\}\\\#I=j}}\!\!\!\!\!\!\!\! \#(X/\langle I {\cup} K\rangle)(\FF_q).
\end{align*}
Moreover, given $I$ a subset of $(\ZZ/2\ZZ)^r$, then $\#(X/\langle I\rangle)(\FF_q)$ can be written as integral linear combination of $\#(X/\langle S\rangle)(\FF_q)$ with $S$ varying among all the subsets of $B$.
\end{cor}
\begin{proof}
It follows by \Cref{lem:J_ind_step} and \Cref{th:jac_formula} observing that if an abelian variety $J$ is isogenous to $J_1
\times J_2$, then the trace of Frobenius satisfies $\mathrm{tr}(\mathrm{Frob})=\mathrm{tr}(\mathrm{Frob_1})+\mathrm{tr}(\mathrm{Frob_2})$, where, for a prime power $q$, we denote by $\mathrm{Frob}$, $\mathrm{Frob}_1$ and $\mathrm{Frob}_2$ the $q$-th Frobenius on $J$, $J_1$ and $J_2$ respectively.
\end{proof}

\begin{ese}\label{ese:genus_appendix}
Suppose that $X$ is (a quotient of) a curve $X(n_0,n_{\ns})$ and let $\{q_1,\ldots,q_s\}$ be a set of coprime prime powers dividing exactly the level $n_0n_{\ns}$ of $X$. If we denote by $g(\cdot)$ the genus, we have
\[
g(X/\langle W_{q_1\cdots q_s}\rangle)= \tfrac{1-(-1)^{s+1}}{2}g(X)+(-1)^{s+1}\sum_{j=1}^s (-2)^{j-1} \!\!\!\!\!\!\!\!\sum_{\substack{I\subseteq \{W_{q_1},\ldots,W_{q_s}\}\\\#I=j}}\!\!\!\!\!\!\!\! g(X/\langle I \rangle).
\]
Using this formula we compute once again the genus of the record curve $X/\langle W_6\rangle$, with $X=X(6,7)/\langle W_7\rangle$. From the algorithm in \cite{DLMS23} we get: $g(X)=16$, $g(X/\langle W_2\rangle)=6$, $g(X/\langle W_3\rangle)=7$ and $g(X/\langle W_2,W_3\rangle)=2$.  
The above formula gives
$$g(X/\langle W_6\rangle)=g(X)-g(X/\langle W_2\rangle)-g(X/\langle W_3\rangle)+2g(X/\langle W_2,W_3\rangle)=16-6-7+4=7 ,$$
in accordance with Example \ref{ex_max} and \Cref{sc:computing-equation}.
\end{ese}

\section{Tables for the number of points}\label{app:bestresults}

In the following tables (\ref{tab2A}, \ref{tab2B}, \ref{tab3A}, \ref{tab5A}, \ref{tab7A}, \ref{tab11A}, \ref{tab13A}, \ref{tab17A}, \ref{tab19A}) we list the quotient curves $X(n_0, n_\ns)/\langle W_{d_1},\ldots,W_{d_k}\rangle$ achieving the largest number of points over each pair (genus, finite field). All items improving the actual best known bounds in \cite{manypoints} are in bold.

We explain the notation of the tables which was used to save space. 
Let
\[
n_0n_\ns = \prod_i p_i^{v_i} = \prod_i q_i
\]
be the prime factorization of $n_0n_\ns$, with $p_1<p_2< \ldots$ and $q_i = p_i^{v_i}$.
Then each exact divisor $d_h$ is uniquely the product of certain $q_{j}$'s and, in the tables, we denote $X(n_0, n_\ns)/\langle W_{d_1},\ldots,W_{d_k}\rangle$ as 
\[
(n_0,n_\mathrm{ns})\, \{ i_{1,1},\ldots, i_{1,m_1};\ldots \ldots \ldots; i_{k,1},\ldots, i_{k,m_k}\}
\]
where for each $h=1,\ldots, k$, 
\[
d_h=\prod_{j=1}^{m_h} q_{i_{h,j}}.
\]
As an example the entry ``$(945,1)\{2;1,3\},49$", in the table below corresponding to row $g=27$ and column $\mathbb{F}_{2^2}$, means that the quotient of $X(945,1)$ by $\langle W_5,W_{189}\rangle$ has 49 points over $\mathbb{F}_{2^2}$, where $d_1=5=p_2$ and $d_2=189=p_1^3 p_3$ since, in this case, $p_1=3$, $p_2=5$ and $p_3=7$.

Another example is the entry ``$(418, 3)\{2; 1, 3; 3, 4\},4261$", in Table \ref{tab5A}, row $g=41$ and column $\mathbb{F}_{5^5}$, that means that the quotient of $X(418,3)$ by $\langle W_3,W_{22},W_{209}\rangle$ has 4261 points over $\mathbb{F}_{5^5}$, where $d_1=3=p_2$ and $d_2=22=p_1 p_3$ and $d_3=209=p_3 p_4$ since, in this case, $p_1=2$, $p_2=3$, $p_3=11$ and $p_4=19$.

\tiny

\oddsidemargin=-1.9cm
\evensidemargin=-1.9cm
\textheight 24 cm

\newpage
\begin{table}
\[
\begin{array}{cccc}
g & \mathbb{F}_{2} & \mathbb{F}_{2^2} & \mathbb{F}_{2^3}\\
\toprule
 1  & (11,1)\{\},5  & (1,11)\{ 1 \},9  & (1,7)\{\},14 \\
 2  & (67,1)\{ 1 \},6  & (23,1)\{\},10  & (35,1)\{ 2 \},16   \\
 3  & (97,1)\{ 1 \},7  & (17,5)\{ 1 ; 2 \},14  & (3,7)\{ 1, 2 \},16   \\
 4  & (47,3)\{ 1 ; 2 \},7  & (1,11)\{\},15  & (119,1)\{ 2 \},16   \\
 5  & (145,1)\{ 1, 2 \},8  & (209,1)\{ 1, 2 \},17  & (3,7)\{ 2 \},17    \\
 6  & (223,1)\{ 1 \},9  & (299,1)\{ 1, 2 \},20  & (9,7)\{ 1 ; 2 \},17   \\
 7  & (187,1)\{ 1, 2 \},10  & (189,1)\{ 1, 2 \},21  & (35,3)\{ 1, 2, 3 \},16    \\
\end{array}
\]
\caption{Table for $\max\{|X(n_0,n_{\mathrm{ns}})/\langle W_{d_1},\ldots,W_{d_t}\rangle(\FF_q)|\}$  with $q=2^k$, $k=1,2,3$ and
$n_0 n_{\textnormal{ns}}^2 \le 10^4$}\label{tab2A}
\end{table}

\newpage
\begin{table}
\[
\begin{array}{ccccc}
g  &\mathbb{F}_{2^4}  &\mathbb{F}_{2^5} &\mathbb{F}_{2^6} &\mathbb{F}_{2^7} \\
\toprule
 1  & (3,5)\{ 1 \},25  & (11,3)\{ 1 ; 2 \},44  & (1,11)\{ 1 \},81  & (11,1)\{\},145   \\
 2    & (1,9)\{\},31  & (13,3)\{ 2 \},51  & (3,7)\{ 1 ; 2 \},91  & (23,1)\{\},172   \\
 3   & (43,1)\{\},33  & (11,3)\{ 2 \},63  & (105,1)\{ 1 ; 2 \},103  & (33,1)\{\},174   \\
 4  & (69,1)\{ 1 \},32  & (61,1)\{\},53  & (1,15)\{ 1 \},119  & (29,3)\{ 1, 2 \},215   \\
 5   & (555,1)\{ 1 ; 3 ; 2 \},33  & (13,3)\{\},62  & (35,3)\{ 1, 3 ; 2, 3 \},129  & (213,1)\{ 2 \},198   \\
 6   & (111,1)\{ 2 \},45  & (297,1)\{ 1 ; 2 \},66  & (117,1)\{ 2 \},143  & (153,1)\{ 2 \},199   \\
 7    & (135,1)\{ 1 \},39  & (423,1)\{ 1 ; 2 \},56  & (1,15)\{\},110  & (69,1)\{\},228   \\
\end{array}
\]
\caption{Table for $\max\{|X(n_0,n_{\mathrm{ns}})/\langle W_{d_1},\ldots,W_{d_t}\rangle(\FF_q)|\}$  with $q=2^k$, $k=4,5,6,7$ and
$n_0 n_{\textnormal{ns}}^2 \le 10^4$}\label{tab2B}
\end{table}

\newpage
\begin{table}
\[
\begin{array}{cccccc}
g & \mathbb{F}_{3} & \mathbb{F}_{3^2} & \mathbb{F}_{3^3} &\mathbb{F}_{3^4}  &\mathbb{F}_{3^5} \\
\toprule
  1  & (23,2)\{ 1, 2 \},7  & (1,7)\{\},16  & (1,10)\{ 2 \},38  & (1,10)\{ 2 \},96  & (1,11)\{ 1 \},275   \\
 2  & (28,1)\{\},8  & (2,5)\{\},20  & (11,4)\{ 2 \},44  & (23,1)\{\},116  & (7,5)\{ 1, 2 \},306   \\
 3  & (52,1)\{ 2 \},10  & (64,1)\{\},28  & (29,2)\{ 2 \},46  & (25,2)\{\},124  & (4,5)\{ 2 \},304   \\
 4  & (74,1)\{ 2 \},11  & (14,5)\{ 1 ; 2, 3 \},30  & (41,2)\{ 2 \},44  & (80,1)\{ 2 \},138  & (85,1)\{ 2 \},346   \\
 5  & (106,1)\{ 2 \},12  & (2,7)\{\},32  & (110,1)\{ 3 \},52  & (46,1)\{\},132  & (14,5)\{ 2 ; 3 \},366   \\
 6  & (170,1)\{ 2 ; 3 \},13  & \mathbf{(398,1)\{ 1 ; 2 \},37}  & (8,5)\{ 1 \},56  & (95,2)\{ 2, 3 \},142  & (121,1)\{\},364   \\
 7  & (266,1)\{ 1 ; 2, 3 \},13  & (238,1)\{ 2, 3 \},40  & (62,1)\{\},54  & (122,1)\{ 1 \},137  & (380,1)\{ 2 ; 3 \},366   \\
\end{array}
\]
\caption{\mbox{Table for $\max\{|X(n_0,n_{\mathrm{ns}})/\langle W_{d_1},\ldots,W_{d_t}\rangle(\FF_q)|\}$  with $q=3^k$, $k=1,2,3,4,5$} and
$n_0 n_{\textnormal{ns}}^2 \le 10^4$}\label{tab3A}
\end{table}

\newpage
\begin{table}
\[
\begin{array}{cccccc}
g & \mathbb{F}_{5} & \mathbb{F}_{5^2} & \mathbb{F}_{5^3} &\mathbb{F}_{5^4}  &\mathbb{F}_{5^5} \\
\toprule
 1  &  (7,4)\{ 1  ; 2 \},10  &  (1,6)\{ 2 \},36  &  (1,8)\{\},148  &  (1,11)\{ 1 \},675  &  (13,4)\{ 1  ; 2 \},3227   \\
 2  &  (11,6)\{ 3  ; 1, 2 \},12  &  (1,12)\{ 2 \},46  &  (14,3)\{ 1  ; 3 \},170  &  (57,1)\{ 1, 2 \},724  &  (17,6)\{ 3  ; 1, 2 \},3328   \\
 3  &  (76,1)\{ 2 \},15  &  (19,12)\{ 1  ; 3  ; 2 \},56  &  (14,3)\{ 3 \},192  &  (29,2)\{ 2 \},738  &  (13,4)\{ 1 \},3404   \\
 4  &  (86,1)\{ 2 \},17  &  (16,3)\{ 2 \},66  &  (132,1)\{ 1, 3  ; 2, 3 \},188  &  (67,2)\{ 1, 2 \},774  &  (62,3)\{ 1  ; 2, 3 \},3530   \\
 5  &  (21,4)\{ 1  ; 3 \},16  & \mathbf{ (209,2)\{ 3  ; 1, 2 \},71}  &  (186,1)\{ 1, 3  ; 2, 3 \},184  &  (133,1)\{ 1, 2 \},823  &  (209,2)\{ 3  ; 1, 2 \},3631   \\
 6  &  (163,1)\{ 1 \},19  &  (84,1)\{ 1, 3 \},74  &  (112,1)\{ 1 \},178  &  (117,1)\{ 2 \},902  &  (84,1)\{ 3 \},3612   \\
 7  &  (44,3)\{ 2  ; 3 \},21  &  (16,3)\{\},84  &  (28,3)\{ 3 \},240  &  (232,1)\{ 1  ; 2 \},854  &  (231,1)\{ 1  ; 3 \},3644   \\
\end{array}
\]
\caption{\mbox{Table for $\max\{|X(n_0,n_{\mathrm{ns}})/\langle W_{d_1},\ldots,W_{d_t}\rangle(\FF_q)|\}$  with $q=5^k$, $k=1,2,3,4,5$} and
$n_0 n_{\textnormal{ns}}^2 \le 10^4$}\label{tab5A}
\end{table}

\newpage
\begin{table}
\[
\begin{array}{cccccc}
g & \mathbb{F}_{7} & \mathbb{F}_{7^2} & \mathbb{F}_{7^3} &\mathbb{F}_{7^4}  &\mathbb{F}_{7^5} \\
\toprule
 1  &  (1,15) \{ 1 ; 2 \},13  &  (1,8) \{\},64  &  (31,2) \{ 2 \},380  &  (1,6) \{ 2 \},2496  &  (1,10) \{ 2 \},17050   \\
 2  &  (1,15) \{ 1, 2 \},16  &  (1,16) \{ 1 \},78  &  (1,9) \{\},412  &  (1,12) \{ 2 \},2590  &  (8,5) \{ 1 ; 2 \},17292   \\
 3  &  (13,6) \{ 1, 2, 3 \},18  &  (9,4) \{ 1 \},92  &  (33,2) \{ 3 \},446  &  (29,2) \{ 2 \},2684  &  (4,5) \{ 2 \},17534   \\
 4  &  (137,1) \{ 1 \},21  &  (54,1) \{\},102  &  (165,1) \{ 2 ; 3 \},440  &  (16,3) \{ 2 \},2778  &  (22,3) \{ 1 ; 3 \},17770   \\
 5  &  (2,15) \{ 2 ; 3 \},24  &  (52,1) \{\},112  &  (88,1) \{ 2 \},440  &  (29,2) \{\},2808  &  (11,3) \{\},18012   \\
 6  &  (73,3) \{ 1 ; 2 \},26  &  (156,1) \{ 2, 3 \},122  &  (33,4) \{ 2 ; 3 \},480  &  (102,1) \{ 3 \},2774  &  (190,1) \{ 2, 3 \},17830   \\
 7  &  (257,1) \{ 1 \},26  &  (1,16) \{\},148  &  (20,3) \{ 2 \},474  &  (87,2) \{ 1, 2, 3 \},2954  &  (8,5) \{ 2 \},18496   \\
\end{array}
\]
\caption{\mbox{Table for $\max\{|X(n_0,n_{\mathrm{ns}})/\langle W_{d_1},\ldots,W_{d_t}\rangle(\FF_q)|\}$  with $q=7^k$, $k=1,2,3,4,5$} and
$n_0 n_{\textnormal{ns}}^2 \le 10^4$}\label{tab7A}
\end{table}

\newpage
\begin{table}
\[
\begin{array}{cccccc}
g & \mathbb{F}_{11} & \mathbb{F}_{11^2} & \mathbb{F}_{11^3} &\mathbb{F}_{11^4}  &\mathbb{F}_{11^5} \\
\toprule
 1  & (10,3)\{ 2 ; 3 \} ,18  & (1,6)\{ 2 \} ,144  & (19,1)\{\} ,1404  & (19,2)\{ 2 \} ,14875  & (5,12)\{ 1 ; 3 ; 2 \} ,161854   \\
 2  & (91,2)\{ 2 ; 1, 3 \} ,22  & (1,12)\{ 2 \} ,166  & (48,1)\{ 2 \} ,1468  & (161,1)\{ 1 ; 2 \} ,15118  & (2,7)\{ 2 \} ,162656   \\
 3  & (109,1)\{ 1 \} ,25  & (5,6)\{ 2 \} ,188  & (105,2)\{ 4 ; 1, 3 \} ,1516  & (27,4)\{ 1 ; 2 \} ,15287  & (52,1)\{ 2 \} ,163458   \\
 4  & (74,1)\{ 2 \} ,27  & (1,15)\{ 1 \} ,210  & (86,1)\{ 2 \} ,1600  & (111,2)\{ 2 ; 1, 3 \} ,15520  & (39,2)\{ 3 \} ,164260   \\
 5  & (181,1)\{ 1 \} ,32  & (1,30)\{ 2 ; 3 \} ,232  & (13,6)\{ 1, 3 \} ,1516  & (378,1)\{ 1 ; 3 ; 2 \} ,15699  & (3,7)\{ 2 \} ,165062   \\
 6  & (122,1)\{ 2 \} ,32  & (103,3)\{ 1 ; 2 \} ,239  & (1,21)\{ 2 \} ,1622  & (84,1)\{ 3 \} ,15590  & (78,1)\{ 3 \} ,165864   \\
 7  & (229,1)\{ 1 \} ,37  & (1,24)\{ 2 \} ,276  & (65,2)\{ 1, 3 \} ,1590  & (37,6)\{ 2 ; 1, 3 \} ,15690  & \mathbf{(6,7)\{ 3 ; 1, 2 \} ,166666}   \\
\end{array}
\]
\caption{\mbox{Table for $\max\{|X(n_0,n_{\mathrm{ns}})/\langle W_{d_1},\ldots,W_{d_t}\rangle(\FF_q)|\}$  with $q=11^k$, $k=1,2,3,4,5$} and
$n_0 n_{\textnormal{ns}}^2 \le 10^4$}\label{tab11A}
\end{table}

\newpage
\begin{table}
\[
\begin{array}{cccccc}
g & \mathbb{F}_{13} & \mathbb{F}_{13^2} & \mathbb{F}_{13^3} &\mathbb{F}_{13^4}  &\mathbb{F}_{13^5} \\
\toprule
 1  & (1,21)\{ 1 ; 2\},21  & (1,7)\{\},196  & (2,5)\{ 2\},2290  & (1,15)\{ 1 ; 2\},28899  & (1,10)\{ 2\},372496   \\
 2  & (1,30)\{ 2 ; 1, 3\},26  & (2,7)\{ 2\},222  & (22,1)\{\},2382  & (17,6)\{ 3 ; 1, 2\},29236  & (11,6)\{ 3 ; 1, 2\},373698   \\
 3  & (43,3)\{ 1 ; 2\},29  & (2,9)\{ 1, 2\},245  & (5,6)\{ 2\},2408  & (43,3)\{ 1 ; 2\},29573  & (9,4)\{ 1\},374900   \\
 4  & (214,1)\{ 1 ; 2\},31  & (133,1)\{ 2\},270  & (5,12)\{ 1, 3 ; 2, 3\},2500  & (86,1)\{ 2\},29840  & (126,1)\{ 3 ; 1, 2\},376094   \\
 5  & (11,5)\{ 1, 2\},34  & (266,1)\{ 3 ; 1, 2\},295  & (56,3)\{ 2 ; 3\},2548  & (69,2)\{ 1, 3\},30078  & (96,1)\{ 2\},377296   \\
 6  & (67,3)\{ 1 ; 2\},41  & (97,2)\{ 1, 2\},310  & (220,1)\{ 2 ; 1, 3\},2508  & (203,1)\{ 1, 2\},29944  & (11,6)\{ 1, 2\},376094   \\
 7  & (43,3)\{ 1, 2\},37  & (4,9)\{ 1, 2\},339  & (5,8)\{ 2\},2600  & (50,3)\{ 2 ; 3\},30426  & (15,4)\{ 1\},377304   \\
\end{array}
\]
\caption{\mbox{Table for $\max\{|X(n_0,n_{\mathrm{ns}})/\langle W_{d_1},\ldots,W_{d_t}\rangle(\FF_q)|\}$  with $q=13^k$, $k=1,2,3,4,5$} and
$n_0 n_{\textnormal{ns}}^2 \le 10^4$}\label{tab13A}
\end{table}

\newpage
\begin{table}
\[
\begin{array}{cccccc}
g & \mathbb{F}_{17} & \mathbb{F}_{17^2} & \mathbb{F}_{17^3} &\mathbb{F}_{17^4}  &\mathbb{F}_{17^5} \\
\toprule
 1  & (190,1)\{ 1 ; 3 ; 2 \} ,25  & (1,6)\{ 2 \} ,324  & (7,6)\{ 3 ; 1, 2 \} ,5054  & (1,10)\{ 2 \} ,84096  & (1,14)\{ 1, 2 \} ,1422141   \\
 2  & (5,4)\{ 2 \} ,30  & (1,12)\{ 2 \} ,358  & (87,1)\{ 2 \} ,5166  & (3,10)\{ 2 ; 3 \} ,84670  & (13,4)\{ 1, 2 \} ,1424424   \\
 3  & (127,1)\{ 1 \} ,36  & (71,2)\{ 2 \} ,392  & (4,5)\{ 2 \} ,5256  & (20,3)\{ 1, 3 ; 2, 3 \} ,85244  & (33,1)\{\} ,1426584   \\
 4  & (80,1)\{ 2 \} ,42  & (16,3)\{ 2 \} ,426  & (70,1)\{ 2 \} ,5350  & (60,1)\{ 3 \} ,85818  & (159,2)\{ 1, 3 ; 2, 3 \} ,1429140   \\
 5  & (23,6)\{ 3 ; 1, 2 \} ,38  & (7,12)\{ 2 ; 1, 3 \} ,448  & (155,1)\{ 2 \} ,5390  & (20,3)\{ 1, 2 \} ,85496  & (52,1)\{\} ,1430904   \\
 6  & (46,3)\{ 2 ; 3 \} ,46  & (22,3)\{ 1, 2, 3 \} ,470  & (47,4)\{ 1 ; 2 \} ,5372  & (24,5)\{ 1 ; 3 ; 2 \} ,86345  & (104,1)\{ 2 \} ,1432818   \\
 7  & (135,2)\{ 3 ; 1, 2 \} ,43  & (16,3)\{\} ,516  & (55,3)\{ 1 ; 3 \} ,5572  & (365,1)\{ 1 ; 2 \} ,86577  & (215,2)\{ 1, 3 ; 2, 3 \} ,1432388   \\
\end{array}
\]
\caption{\mbox{Table for $\max\{|X(n_0,n_{\mathrm{ns}})/\langle W_{d_1},\ldots,W_{d_t}\rangle(\FF_q)|\}$  with $q=17^k$, $k=1,2,3,4,5$} and
$n_0 n_{\textnormal{ns}}^2 \le 10^4$}\label{tab17A}
\end{table}

\newpage
\begin{table}
\[
\begin{array}{cccccc}
g & \mathbb{F}_{19} & \mathbb{F}_{19^2} & \mathbb{F}_{19^3} &\mathbb{F}_{19^4}  &\mathbb{F}_{19^5} \\
\toprule
 1  &  (1,24)\{ 1  ; 2 \}  ,28  &  (1,7)\{\}  ,400  &  (1,12)\{ 1 \}  ,7024  &  (11,3)\{ 1  ; 2 \}  ,131040  &  (23,2)\{ 1, 2 \}  ,2478982   \\
 2  &  (130,1)\{ 1, 3  ; 2, 3 \}  ,34  &  (22,1)\{\}  ,438  &  (30,1)\{ 2 \}  ,7188  &  (5,7)\{ 1  ; 2 \}  ,131758  &  (191,1)\{ 1 \}  ,2482286   \\
 3  &  (149,1)\{ 1 \}  ,38  &  (33,1)\{\}  ,476  &  (9,4)\{ 1 \}  ,7352  &  (23,4)\{ 1  ; 2 \}  ,132476  &  (46,1)\{ 1 \}  ,2484746   \\
 4  &  (305,1)\{ 1  ; 2 \}  ,39  &  (66,1)\{ 1, 2 \}  ,514  &  (143,2)\{ 1, 3  ; 2, 3 \}  ,7333  &  (130,1)\{ 2  ; 3 \}  ,133194  &  (262,1)\{ 1  ; 2 \}  ,2487889   \\
 5  &  (33,4)\{ 1, 3  ; 2, 3 \}  ,44  &  (55,1)\{\}  ,536  &  (9,4)\{\}  ,7352  &  (55,6)\{ 3  ; 4  ; 1, 2 \}  ,133048  &  (46,1)\{\}  ,2490510   \\
 6  &  (221,1)\{ 1, 2 \}  ,48  &  (73,3)\{ 1  ; 2 \}  ,560  &  (58,1)\{\}  ,7420  &  (104,1)\{ 2 \}  ,133478  &  (103,3)\{ 1  ; 2 \}  ,2490728   \\
 7  &  (298,1)\{ 1  ; 2 \}  ,47  &  (215,2)\{ 1, 3  ; 2, 3 \}  ,600  &  (135,1)\{ 1 \}  ,7464  &  (65,2)\{ 1, 2 \}  ,134460  &  (11,10)\{ 3  ; 1, 2 \}  ,2492010   \\
\end{array}
\]
\caption{\mbox{Table for $\max\{|X(n_0,n_{\mathrm{ns}})/\langle W_{d_1},\ldots,W_{d_t}\rangle(\FF_q)|\}$ with $q=19^k$, $k=1,2,3,4,5$} and
$n_0 n_{\textnormal{ns}}^2 \le 10^4$}\label{tab19A}
\end{table}

\clearpage
\normalsize

\section{Tables for the maximal curves}\label{app:maxcurves}

In the following tables (\ref{tabMax2}, \ref{tabMax34}, \ref{tabMax>4}), we list the quotient curves $X(n_0, n_\ns)/\langle W_{d_1},\ldots,W_{d_k}\rangle$ with genus $g\ge 2$ achieving the maximal number of points for a certain pair (genus, field), as discussed in \cref{sc:other}. The curves are grouped by their real Weil polynomial $h_W(t)$ that is displayed as well.
The notation for the curves is the same as in the tables of \Cref{app:bestresults}. The two new maximal curves of \cref{sc:computing-equation} are in bold.

\tiny

\begin{table}[H]
\[
\begin{array}{llll}
g & q & h_W(x) & \text{Curves}  \\
\toprule
 2 & 2 & x^2+3x+1 & (67, 1) \{ 1 \} , \,\, (73, 1) \{ 1 \} , \,\, (93, 1) \{ 1 \} , \,\, (103, 1) \{ 1 \} , \,\, (115, 1) \{ 1 \} , \,\, (133, 1) \{ 1 \}  \\  
  \hline 
 2 & 2^2 & x^2+5x+9 & (23, 1) \{\}, \,\, (29, 3) \{ 1 ; 2 \} , \,\,(31, 1)\{\} , \,\, (31, 3) \{ 1, 2 \} , \,\, (69, 1) \{ 1, 2 \} , \,\,(87, 1) \{ 2 \} , \,\, (107, 1) \{ 1 \} , \,\, (125, 1) \{ 1 \}  , \,\,(161, 1) \{ 1 ; 2 \}  ,  \\
    &  & &  (167, 1) \{ 1 \} , \,\, (177, 1) \{ 1 ; 2 \} , \,\, (191, 1) \{ 1 \},\,\,(205, 1) \{ 1 ; 2 \}  , \,\, (213, 1) \{ 1 ; 2 \}, \,\, (221, 1) \{ 1 ; 2 \} , \,\, (287, 1) \{ 1 ; 2 \}, \,\,  (299, 1) \{ 1 ; 2 \}   \\
  \hline 
 2 & 2^7 & x^2 + 43x + 713 & (23, 1)\{\} , \,\, (29, 3)\{ 1 ; 2 \}, \,\, (31, 3)\{ 1, 2 \} , \,\, (69, 1)\{ 1, 2 \} , \,\, (107, 1)\{ 1 \} , \,\, (125, 1)\{ 1 \} , \,\, (161, 1)\{ 1 ; 2 \} , \,\,(167, 1)\{ 1 \} ,  \\
    &  & & (177, 1)\{ 1 ; 2 \} , \,\,(191, 1)\{ 1 \},\,\, (205, 1)\{ 1 ; 2 \}, \,\, (213, 1)\{ 1 ; 2 \}, \,\, (221, 1)\{ 1 ; 2 \} , \,\, (287, 1)\{ 1 ; 2 \} \\
  \hline 
 2 & 3 & x^2 +4x + 2 & (85, 1)\{ 1 ;2 \}, \,\, (85, 2)\{ 1 ; 2 ; 3 \}, \,\, (170, 1)\{ 1 ; 2;3 \} \\
 \cline{3-4} 
    &  & x^2 +4x + 3 & (37, 2)\{ 1, 2 \}, \,\, (53, 2)\{ 1, 2 \}, \,\, (58, 1)\{ 2 \} , \,\, (77, 1)\{ 1, 2 \} , \,\, (77, 2)\{ 1 ; 2, 3 \} , \,\, (77, 2)\{ 2; 1, 3 \} , \,\, (106, 1)\{ 1 ; 2 \} , \,\, (116, 1)\{ 1; 2 \} \\
 \cline{3-4} 
    &  & x^2 +4x + 4 & (28, 1)\{\} , \,\, (40, 1)\{ 2 \} , \,\, (43, 2)\{ 1, 2 \} , \,\, (65, 2)\{ 2 ; 1, 3 \} , \,\, (122, 1)\{ 1 ; 2 \} \\
  \hline 
 2 & 3^2 & x^2+10x+37 & (2, 5)\{\} , \,\, (4, 5)\{ 1, 2 \} ,\,\, (4, 5)\{ 1 \} ,  \,\,(7, 5)\{ 1, 2 \}  , \,\, (7, 10)\{ 1 ; 2, 3 \} , \,\, (11, 2)\{\} , \,\,(11, 4)\{ 2 \} , \,\,  (13, 4)\{ 1, 2 \}, \,\, (14, 5)\{ 1 ;2 ;3 \} ,  \\
    &  & & (22, 1)\{\},\,\,(29, 4)\{ 1 ;2 \} , \,\, (35, 2)\{ 1, 2 \} \,\, (44, 1)\{ 1 \} , \,\,(50, 1)\{\} , \,\, (89, 2)\{ 1, 2 \} , \,\,  (100, 1)\{ 1 \} , \,\,(103, 1)\{ 1 \} , \,\, (103, 2)\{ 1 ; 2 \}, \\
    &  & & (110, 1)\{ 1, 3 ; 2, 3 \} , \,\, (115, 1)\{ 1;2 \} , \,\, (115, 2)\{ 1 ;2;3 \} , \,\, (121, 1)\{ 1 \} , \,\, (121, 2)\{ 1 ; 2 \} , \,\,  (142, 1)\{ 2 \} , \,\,(143, 1)\{ 1, 2 \} ,  \\
    &  & &  (143, 2)\{ 1;2, 3 \} ,\,\,(158, 1)\{ 1 ; 2 \}, \,\, (161, 1)\{ 1 ; 2 \} , \,\, (161, 2)\{ 1 ;2;3 \}, \,\, (166, 1)\{ 1 ; 2 \} , \,\, (184, 1)\{ 1 ; 2 \} , \,\,(190, 1)\{ 2 ; 3 \} , \\
    &  & &(191, 1)\{ 1 \} , \,\, (191, 2)\{ 1; 2 \}, \,\, (205, 1)\{ 1 ;2 \}, \,\, (205, 2)\{ 1 ;2;3 \}, \,\, (206, 1)\{ 1; 2 \} , \,\, (230, 1)\{ 1 ;2;3 \} , \,\, (284, 1)\{ 1 ; 2 \}, \\
    &  & & (286, 1)\{ 1;2;3 \} , \,\, (380, 1)\{ 1;2;3 \} \\
  \hline 
 2 & 3^5 & x^2+62 x +1441 & (7, 5)\{ 1, 2 \} , \,\, (7, 10)\{ 1 ;2, 3 \}, \,\, (13, 4)\{ 1, 2 \}, \,\, (14, 5)\{ 1 ;2;3 \} , \,\, (22, 1)\{\} , \,\, (44, 1)\{ 1 \} , \,\, (89, 2)\{ 1, 2 \}, \,\, (103, 1)\{ 1 \} ,  \\
    &  & & (103, 2)\{ 1 ;2 \},\,\,(110, 1)\{ 1, 3 ; 2, 3 \} , \,\, (115, 1)\{ 1 ; 2 \}, \,\, (115, 2)\{ 1;2;3 \} , \,\,(121, 1)\{ 1 \} , \,\, (121, 2)\{ 1 ; 2 \} , \,\,(143, 1)\{ 1, 2 \} ,  \\
    &  & & (143, 2)\{ 1; 2, 3 \} , \,\, (158, 1)\{ 1 ;2 \},\,\, (161, 1)\{ 1 ;2 \}, \,\, (161, 2)\{ 1 ;2;3 \}, \,\, (166, 1)\{ 1 ; 2 \}, \,\, (190, 1)\{ 2;3 \}, \,\, (191, 1)\{ 1 \} ,  \\
    &  & &  (191, 2)\{ 1 ; 2 \} , \,\, (205, 1)\{ 1; 2 \}, \,\, (205, 2)\{ 1 ;2;3 \},\,\, (206, 1)\{ 1 ;2 \} , \,\, (230, 1)\{ 1 ;2;3 \} , \,\, (286, 1)\{ 1 ;2;3 \}, \,\, (380, 1)\{ 1 ;2;3 \} \\
  \hline 
 2 & 5 & x^2+6x+8 & (11, 6)\{ 3 ;1, 2 \} , \,\, (21, 4)\{ 2 ; 1, 3 \}, \,\, (22, 3)\{ 1; 2, 3 \} , \,\, (37, 2)\{ 1, 2 \} , \,\, (102, 1)\{ 1; 2, 3 \} , \,\,(129, 1)\{ 1 ; 2 \} , \,\, (129, 2)\{ 1;2;3 \} \\
 \cline{3-4} 
    &  & x^2+6x+9 & (57, 2)\{ 2; 1, 3 \}  , \,\,(61, 2)\{ 1, 2 \} , \,\, (67, 1)\{ 1 \} , \,\, (67, 2)\{ 1; 2 \}, \,\, (88, 1)\{ 1 ; 2 \} , \,\, (91, 1)\{ 1, 2 \} , \,\, (91, 2)\{ 1 ; 2, 3 \} , \,\, (134, 1)\{ 1 ;2 \} \\
  \hline 
 2 & 5^2 & x^2+20x+140 & (1, 12)\{ 2 \}  , \,\, (2, 7)\{ 1 \} , \,\, (7, 6)\{ 2, 3 \}, \,\, (8, 3)\{ 2 \} , \,\, (13, 3)\{ 1 \} , \,\,(13, 6)\{ 1 ; 2 \}, \,\,(27, 2)\{ 2 \}, \,\, (28, 1)\{\}, \,\, (28, 3)\{ 1, 3; 2, 3 \} , \\
    &  & & (47, 2)\{ 2 \}, \,\,(51, 4)\{ 1;2;3 \} , \,\, (98, 1)\{ 1, 2 \} \\
  \hline 
 2 & 5^3 & x^2+44x+724 & (14, 3)\{ 1; 3 \}  \\
  \hline 
 2 & 7 & x^2+8x+15 & (1, 15)\{ 1, 2 \} , \,\, (1, 30)\{ 1 ;2, 3 \} , \,\,(57, 2)\{ 3; 1, 2 \} \\
 \cline{3-4} 
    &  & x^2+8x+16 & (8, 3)\{ 2 \}  \\
  \hline 
 2 & 7^2 & x^2+28x+280 & (1, 16)\{ 1 \}  , \,\, (30, 1)\{ 2 \}, \,\, (48, 1)\{ 2 \} , \,\, (48, 1)\{ 1 \}, \,\, (60, 1)\{ 1; 2 \} \\
  \hline 
 2 & 11^2 & x^2+44x+704 & (1, 12)\{ 2 \} , \,\, (1, 30)\{ 2; 1, 3 \}, \,\, (5, 4)\{ 2 \}, \,\, (5, 6)\{ 2 ; 3 \} , \,\, (7, 4)\{ 2 \}, \,\, (8, 3)\{ 2 \}, \,\, (15, 2)\{ 3 \}, \,\, (21, 4)\{ 3 ;1, 2 \} , \,\,(25, 2)\{ 2 \} ,  \\
    &  & & (27, 2)\{ 2 \}, \,\, (28, 1)\{\},\,\,(40, 1)\{ 2 \} , \,\,(60, 1)\{ 3 ; 1, 2 \} , \,\, (102, 1)\{ 1 ; 2, 3 \}, \,\,(119, 2)\{ 1, 3 ; 2, 3 \} , \,\, (167, 1)\{ 1 \} , \,\, (167, 2)\{ 1 ; 2 \} \\
  \hline 
 2 & 11^5 & x^2+1604x+965284 & (2, 7)\{ 2 \} , \,\, (3, 7)\{ 1; 2 \}, \,\, (3, 14)\{ 1;2;3 \} , \,\, (4, 7)\{ 1 ; 2 \}, \,\, (6, 7)\{ 1 ;2;3\}, \,\, (31, 3)\{ 1, 2 \} , \,\,(31, 6)\{ 1 ;2, 3 \} , \,\, (39, 1)\{ 2 \} ,\,\,(39, 2)\{ 1 ; 3 \}, \\
    &  & &  (78, 1)\{ 1; 3 \}, \,\, (101, 2)\{ 1, 2 \} , \,\, (104, 1)\{ 1 ; 2 \} , \,\, (117, 1)\{ 1; 2 \} , \,\, (117, 2)\{ 1 ;2;3 \} , \,\,(147, 1)\{ 1 ;2 \}, \,\, (147, 2)\{ 1 ;2;3 \} \\
  \hline 
 2 & 13 & x^2+12x+35 & (1, 30)\{ 2; 1, 3 \}  \\
 \cline{3-4} 
    &  & x^2+12x+36 & (107, 1)\{ 1 \} , \,\, (107, 2)\{ 1 ;2 \} \\
  \hline 
 2 & 13^2 & x^2+52x+988 & (2, 7)\{ 2 \} , \,\, (4, 7)\{ 1;2 \} , \,\, (23, 3)\{ 1 ; 2 \} , \,\,(23, 6)\{ 1;2;3 \} , \,\, (46, 3)\{ 1;2;3 \} \\
  \hline 
 2 & 17^2 & x^2+12x+36 & (1, 12)\{ 2 \}, \,\,(1, 30)\{ 2; 1, 3 \}, \,\,(7, 6)\{ 2, 3 \} , \,\,(8, 3)\{ 2 \} , \,\, (27, 2)\{ 2 \} , \,\, (54, 1)\{ 1 \} , \,\, (161, 1)\{ 1 ;2 \} , \,\, (161, 2)\{ 1;2;3 \} ,  \\
    &  & & (191, 1)\{ 1 \} ,\,\,(191, 2)\{ 1 ; 2 \} \\
  \hline 
 2 & 19^2 & x^2+76x+2128 & (22, 1)\{\}, \,\, (29, 3)\{ 1 ;2 \} , \,\,(29, 4)\{ 1 ; 2 \} , \,\, (29, 6)\{ 1 ;2;3 \} , \,\, (33, 1)\{ 1 \} , \,\, (33, 2)\{ 1 ;2 \} , \,\, (44, 1)\{ 1 \} , \,\,(66, 1)\{ 1 ; 2 \} ,  \\
    &  & & (66, 1)\{ 1, 3 ; 2, 3 \},\,\, (66, 1)\{ 2 ; 1, 3 \}, \,\, (121, 1)\{ 1 \} , \,\,(121, 2)\{ 1; 2 \}
  \end{array}
\]
\caption{\mbox{Maximal curves over $\FF_q$ of genus $g=2$.}}\label{tabMax2}
\end{table}

\begin{table}
\[
\begin{array}{llll}
g & q & h_W(x) & \text{Curves}  \\
\toprule
 3 & 2 & x^3 + 4x^2 + 3x - 1 
 & (97, 1)\{ 1 \}  \\
  \hline 
 3 & 2^2 & x^3 + 9x^2 + 33x + 63 
 & (17, 5)\{ 1;2 \} , \,\, (45, 1)\{\}  , \,\, (55, 3)\{ 2 ; 1, 3 \}  , \,\, (61, 3)\{ 1 ;2 \}  , \,\,(63, 1)\{ 2 \} , \,\, (79, 3)\{ 1; 2 \} , \,\, (105, 1)\{ 2, 3 \} , \\
 & & & (195, 1)\{ 1; 2, 3 \}  , \,\, (249, 1)\{ 1; 2 \} \\
  \hline 
 3 & 3 & x^3 + 6x^2 + 9x
 & (52, 1)\{ 2 \}  \\
  \hline 
 3 & 3^2 & x^3 + 18x^2 + 126x + 432
 & (64, 1)\{\}  , \,\, (91, 2)\{ 3 \} , \,\, (124, 1)\{ 2 \} , \,\, (238, 1)\{ 1, 3; 2, 3 \} \\
  \hline 
 3 & 5^2 & x^3 + 30x^2 + 360x + 2200 
 & (19, 12)\{ 1 ;2;3\}  \\
  \hline 
 3 & 5^3 & x^3 + 66x^2 + 1812x + 26488 
 & (14, 3)\{ 3 \}, \,\, (28, 3)\{ 1;3 \} \\
  \hline 
 3 & 7^2 & x^3 + 42x^2 + 714x + 6272 
 & (9, 4)\{ 1 \} , \,\, (45, 1)\{\} , \,\, (45, 2)\{ 1 \}, \,\, (48, 1)\{\}, \,\, (62, 3)\{ 2 ; 3 \}, \,\, (64, 1)\{\} , \,\, (96, 1)\{ 1, 2 \} , \,\, (124, 3)\{ 1;2;3 \} \\
  \hline 
 3 & 11^2 &  x^3 + 66x^2 + 1782x + 25168 
 & (5, 6)\{ 2 \} , \,\, (5, 6)\{ 2, 3 \} , \,\, (15, 4)\{ 3 ;1, 2 \} , \,\, (25, 2)\{\} , \,\,(45, 2)\{ 2 ;3 \} , \,\, (56, 1)\{ 1 \} , \,\, (64, 1)\{\} ,  \\
    &  & & (69, 2)\{ 2 ; 3 \}, \,\,(105, 4)\{ 1; 2 ; 3;4 \}, \,\,(120, 1)\{ 3; 1, 2 \} \\
  \hline 
 3 & 11^5 & x^3 + 2406x^2 + 2412732x + 1290774088 
 & (52, 1)\{ 2 \} , \,\, (78, 1)\{ 3; 1, 2 \} , \,\,(156, 1)\{ 2 ;3 \}, \,\, (312, 1)\{ 1 ;2;3 \} \\
  \hline 
 3 & 17^2 & x^3 + 102x^2 + 4284x + 94792
 & (71, 2)\{ 2 \}  \\
  \hline 
 3 & 19^2 & x^3 + 114x^2 + 5358x + 132848 
 & (33, 1)\{\}  , \,\,(33, 2)\{ 1 \} , \,\, (55, 1)\{ 1 \} , \,\,(55, 2)\{ 1; 2 \} , \,\, (64, 1)\{\}, \,\, (66, 1)\{ 1, 2, 3 \} , \,\, (110, 1)\{ 2 ; 1, 3 \}  \\
  \hline 
 4 & 2^2 & x^4 + 10x^3 + 41x^2 + 96x + 172 
 & (1, 11)\{\}  , \,\, (1, 21)\{ 1, 2 \} , \,\, (25, 3)\{ 2 \} ,\,\,(135, 1)\{ 1, 2 \} , \,\, (225, 1)\{ 1; 2 \}, \,\, (231, 1)\{ 2 ; 1, 3 \}  , \,\,(285, 1)\{ 1 ;2, 3 \} \\
 \cline{3-4} 
    &  & x^4 + 10x^3 + 41x^2 + 100x + 188
    & (7, 15)\{ 1;2;3\} , \,\, (81, 1)\{\} , \,\, (177, 1)\{ 2 \} , \,\, (261, 1)\{ 1 ; 2 \}  \\
 \cline{3-4} 
    &  & x^4 + 10x^3 + 43x^2 + 110x + 205 
    & (29, 3)\{ 1, 2 \} , \,\,(161, 1)\{ 1, 2 \}  , \,\, (251, 1)\{ 1 \} , \,\, (321, 1)\{ 1 ;2 \}, \,\,(341, 1)\{ 1 ; 2 \}, \,\, (483, 1)\{ 1 ;2;3\} \\
  \hline 
 4 & 2^7 & x^4 + 86x^3 + 3275x^2 + 72154x + 1007077 
 & (29, 3)\{ 1, 2 \} , \,\, (161, 1)\{ 1, 2 \} , \,\, (251, 1)\{ 1 \}  , \,\, (321, 1)\{ 1 ; 2 \}  , \,\,(341, 1)\{ 1 ; 2 \}, \,\,(483, 1)\{ 1 ;2;3 \}  \\
  \hline 
 4 & 3^2 & x^4 + 20x^3 + 174x^2 + 860x + 2713 
 & (14, 5)\{ 1 ; 2, 3 \} , \,\, (44, 1)\{\}  , \,\, (110, 1)\{ 2, 3 \} , \,\, (115, 2)\{ 3; 1, 2 \} , \,\, (161, 2)\{ 3 ; 1, 2 \}  ,  \\
 & & &  (220, 1)\{ 1 ; 2, 3 \} , \,\, (286, 1)\{ 1 ; 2, 3 \} \\
  \hline 
 4 & 5^2 & x^4 + 40x^3 + 680x^2 + 6400x + 36800 
 & (16, 3)\{ 2 \}  , \,\, (27, 2)\{\} , \,\, (28, 3)\{ 2, 3 \} , \,\, (98, 1)\{ 1 \} , \,\, (188, 1)\{ 2 \}  \\
  \hline 
 4 & 7^2 & x^4 + 52x^3 + 1176x^2 + 15176x + 123872 
 & (60, 1)\{ 2 \} , \,\, (120, 1)\{ 2; 1, 3 \}  \\
 \cline{3-4} 
    &  & x^4 + 52x^3 + 1182x^2 + 15340x + 125497
    & (54, 1)\{\}  , \,\, (108, 1)\{ 1 \}  \\
  \hline 
 4 & 11^2 & x^4 + 88x^3 + 3344x^2 + 71632x + 950576 
 & (1, 15)\{ 1 \} , \,\, (1, 30)\{ 1; 2 \}, \,\,(15, 4)\{ 2 ; 3 \}  , \,\, (16, 3)\{ 2 \} , \,\,(25, 3)\{ 1 \} , \,\, (25, 6)\{ 1; 2 \} , \,\, (25, 6)\{ 2; 1, 3 \}  , \\
 & & & (27, 2)\{\}  , \,\,(32, 3)\{ 1 ;2 \},\,\,(60, 1)\{ 3 \} , \,\,(80, 1)\{ 2 \}  \\
  \hline 
 4 & 11^5 & x^4 + 3208x^3 + 4503384x^2 +
 & (39, 2)\{ 3 \}  \\
  & & + 3613247392x + 1812300626896  & \\
  \hline 
 4 & 17^2 & x^4 + 136x^3 + 8024x^2 + 268192x + 5571920 
 & (16, 3)\{ 2 \}  , \,\, (27, 2)\{\}, \,\, (31, 6)\{ 2 ; 3 \} , \,\, (54, 1)\{\}  , \,\, (108, 1)\{ 2 \} , \,\, (108, 1)\{ 1, 2 \} , \,\,(108, 1)\{ 1 \}  \\
  \hline 
 4 & 19^2 & x^4 + 152x^3 + 10032x^2 + 375440x + 8739088 
 & (66, 1)\{ 1, 2 \} 
 \end{array}
\]
\caption{\mbox{Maximal curves over $\FF_q$ of genus $3\le g\le 4$.}}\label{tabMax34}
\end{table}

\clearpage

\begin{table}[H]
\[
\begin{array}{llll}
g & q & h_W(x) & \text{Curves}  \\
\toprule
 5 & 2^2 & x^5 + 12x^4 + 62x^3 + 192x^2 + 436x + 864 
 & (209, 1)\{ 1, 2 \} , \,\, (645, 1)\{ 1 ;2;3 \}  \\
  \hline 
 5 & 11^2 & x^5 + 110x^4 + 5390x^3 + 154880x^2 + 2895530x + 37097632
 & (1, 30)\{ 2 ; 3 \} , \,\,(120, 1)\{ 1; 3 \}  \\
  \hline 
 5 & 11^5 & x^5 + 4010x^4 + 7237240x^3 + 7741577680x^2 + 5435351766080x +     2617242553904032
 & (3, 7)\{ 2 \}, \,\, (3, 14)\{ 1 ;3 \}  , \,\, (156, 1)\{ 3 ; 1, 2 \}  \\
  \hline 
 6 & 2^2 & x^6 + 15x^5 + 102x^4 + 425x^3 + 1254x^2 + 2925x + 6021
 & (299, 1)\{ 1, 2 \}  \\
  \hline 
 6 & 11^5 &  x^6 + 4812x^5 + 10614300x^4 + 14191614560x^3 +  & (78, 1)\{ 3 \}  , \,\,(104, 1)\{ 2 \}  , \,\, (156, 1)\{ 1 ; 3 \}  \\
 & & +12809583966480x^2 +  8223105443361792x + 3849697906629612544 & \\
  \hline 
 7 & 2 & x^7 + 7x^6 + 14x^5 - 2x^4 - 30x^3 - 16x^2 + 12x + 8 
 & (187, 1)\{ 1, 2 \}  \\ 
  \hline 
 7 & 2^2 & x^7 + 16x^6 + 116x^5 + 520x^4 + 1685x^3 + 4424x^2 + 10310x + 22880 
 & (189, 1)\{ 1, 2 \}, \,\, (455, 1)\{ 1 ; 2, 3 \}  \\
  \hline 
 7 & 7^2 & x^7 + 98x^6 + 4410x^5 + 120736x^4 + 2248022x^3 + 30194976x^2 + 303918580x     + 2367501248
 & (1, 16)\{\}  \\
  \hline 
 7 & 11^2 & x^7 + 154x^6 + 10934x^5 + 474320x^4 + 14051730x^3 + 301274512x^2 +     4843013868x + 59903475808 
 & (1, 24)\{ 2 \} , \,\,(25, 6)\{ 2 ; 3 \} , \,\,(64, 3)\{ 1 ; 2 \}  \\
  \hline 
 7 & 11^5 & x^7 + 5614x^6 + 14634564x^5 + 23479207640x^4 + 25900587487040x^3 + 
 & \mathbf{(6, 7)\{ 3;1,2 \}}  \\
 & &     +20781817917390432x^2 + 12507502676802272608x + 5736056081814890238208 & \\
  \hline 
 9 & 11^2 & x^9 + 198x^8 + 18414x^7 + 1068672x^6 + 43379226x^5 + 1308681792x^4 +     
 & (1, 30)\{ 2 \}  , \,\, (32, 3)\{ 2 \}  \\
 & & +30454499196x^3 + 560623333248x^2 + 8321568916914x + 101330949451072  & \\
  \hline 
 10 & 2^2 & x^{10} + 22x^9 + 227x^8 + 1488x^7 + 7099x^6 + 26790x^5 + 84949x^4 +     
 & (13, 15)\{ 1 ;2;3\}  \\
 & & +237900x^3 + 614564x^2 + 1518856x + 3689816  & \\
 \cline{3-4} 
    &  & x^{10} + 22x^9 + 227x^8 + 1476x^7 + 6907x^6 + 25398x^5 + 78709x^4 +   
    & (225, 1)\{ 2 \}  \\
    & & +  217680x^3 + 561476x^2 + 1395072x + 3415000  & \\
  \hline 
 10 & 3^2 & x^{10} + 44x^9 + 924x^8 + 12336x^7 + 117822x^6 + 860472x^5 + 5028348x^4 +  
 & (220, 1)\{ 2, 3 \}  \\
 & & +   24382968x^3 + 101458377x^2 + 374775308x + 1272392128 & \\
  \hline 
 10 & 17^2 & x^{10} + 340x^9 + 54740x^8 + 5548800x^7 + 397201600x^6 + 21349618368x^5 +  &  (108, 1)\{\} \\
 & & +    894614065600x^4 + 29960920387200x^3 + 816165461399840x^2  &  \\
  & & + +18327354677661760x + 343126733171222720 & \\
  \hline 
 12 & 11^5 & x^{12} + 9624x^{11} + 44383944x^{10} + 130535252320x^9 + 274862630948400x^8 +     & \mathbf{(156, 1)\{ 3 \}}  \\
 & & + 440993755828543104x^7 + 560170020610594432896x^6 +     577840514161469228879616x^5 + & \\
 & & + 492617240971238592986103360x^4 +     351365791059712247073125157750x^3 +  & \\
 & & +211472505871262576194557806678120x^2 +    108000279404968696858606142571834266x +  & \\
 & & +    46952382003014183461596141477941970132 & \\ 
  \hline 
 13 & 11^2 & x^{13} + 286x^{12} + 39182x^{11} + 3422848x^{10} + 214132490x^9 + 10215137504x^8 +     & (25, 6)\{ 2 \}  \\
 & & + 386406235644x^7 + 11901007829568x^6 + 304192864120002x^5 +     6546901285290080x^4 +   & \\
 & & +120036129906428164x^3 + 1893927675352782784x^2 +  25962335404137817828x +  & \\
 & & + 312310047332950820032 & \\
  \hline 
 19 & 11^2 & x^{19} + 418x^{18} + 84854x^{17} + 11145552x^{16} + 1064837026x^{15} +     78868905456x^{14} + 
 & (64, 3)\{ 2 \} \\
 & & +4713619036588x^{13} + 233624987892384x^{12} +  9795244689058202x^{11} +     & \\
 & & + 352657670048703536x^{10} + 11031136579467494292x^9 +    302625311123193214816x^8 +  & \\
 & & +7338131279130078069588x^7 +     158317948152609796199264x^6 +  & \\
 & & +3056722777699641740353752x^5 +     53095079109227126428981568x^4 +  & \\
 & & +833873235186826014803382474x^3 +     11900141855044287421899255472x^2 +   & \\
 & & +155111515626413902614636636040x +     1856794761749706495782311526412 & 
 \end{array}
\]
\caption{\mbox{Maximal curves over $\FF_q$ of genus $g\ge 5$.}}\label{tabMax>4}
\end{table}

\normalsize
\printbibliography

\end{document}